\documentclass[a4paper,fleqn, 10pt]{article}
\usepackage{hyperref,amsmath,makeidx,amsfonts,a4wide,amssymb,amsthm,longtable, color}
\newtheorem{theorem}{Theorem}

\numberwithin{equation}{section}
\numberwithin{lemma}{section}
\numberwithin{theorem}{section}
\numberwithin{corollary}{section}

\allowdisplaybreaks

\begin{document}
\title{On the discrete analogues of Appell function $F_2$}
\author{Ravi Dwivedi$^{1,}$\footnote{E-mail: dwivedir999@gmail.com}   \, and Vivek Sahai$^{2,}$\footnote{E-mail: sahai\_vivek@hotmail.com (Corresponding author)} \\ ${}^{1}$Department of Science, SAGEMMC, Jagdalpur, Bastar, CG, 494001, India; \\ ${}^{2}$Department of Mathematics and Astronomy, Lucknow University, \\ Lucknow 226007, India.}
\maketitle 
\begin{abstract}
	In this paper, we introduce  two distinct discrete forms of Appell function $F_2$. We  determine their convergence domains, integral representations as well as difference-differential equations that are satisfied by these discrete analogues of $F_2$. We have also obtained the difference-differential formulas, finite and infinite summation formulas, and a complete list of difference-differential recursion relations obeyed by these discrete functions.
	

\end{abstract}
\section{Introduction} 
This paper is the continuation of the research work initiated in \cite{ds16}, where discrete forms of Appell function $F_1$ were discussed. In this paper, we introduce the discrete forms of the second Appell function $F_2$, \cite{emo, sk}:
\begin{align}
	F_2 (a, b, b'; c, c'; x, y) & = \sum_{m, n = 0}^{\infty} \frac{(a)_{m+ n} \, (b)_m  \, (b')_n}{(c)_{m} \, (c')_{n}} \, \frac{x^m \, y^n}{m ! \, n!}, \quad \lvert x \rvert + \lvert y \rvert < 1.
\end{align}
This function has own importance due to its appearance in physical problems 
%In particular .... 
and this wide range of applications has led to the treatment of Appell function $F_2$ in many ways \emph{viz.} basic (or $q$-) version \cite{gr}, finite field analogue \cite{bts}, matrix point of view \cite{ds1, ds5}, matrix variate form \cite{nn}, to name a few. 
 
  Throughout the paper, the set of natural numbers and complex numbers will be denoted by $\mathbb{N}$ and $\mathbb{C}$, respectively. Let  $a$, $b_1$, $b_2$, $c_1$, $c_2$, $t$, $t_1$ and $t_2$ be complex numbers such that $\Re (c_1), \ \Re (c_2) \ne 0, -1, -2, \dots$ and $k$, $k_1$, $k_2 \in \mathbb{N}$. Then, we define two  discrete analogues of Appell  function ${F}_2$ as follows: 
\begin{align}
	\mathcal{F}^{(1)}_2 & = \mathcal{F}^{(1)}_2(a, b_1, b_2; c_1, c_2; t_1, t_2, k_1, k_2, x, y)\nonumber\\ 
	& = \sum_{m,n\geq0} \frac{(a)_{m+n} \, (b_1)_m \, (b_2)_n \, (-1)^{m k_1} \, (-t_1)_{mk_1} \, (-1)^{n k_2} \,  (-t_2)_{nk_2}}{ (c_1)_{m} \, (c_2)_n\, m! \, n!} \ x^m \, y^n; \label{3.1}
\end{align} 
\begin{align}
	\mathcal{F}^{(2)}_2 & =	\mathcal{F}^{(2)}_2(a, b_1, b_2; c_1, c_2; t, k, x, y)\nonumber\\
	& = \sum_{m,n\geq0} \frac{(a)_{m+n} \, (b_1)_m \, (b_2)_n \, (-1)^{(m + n) k} \, (-t)_{(m + n)k}}{(c_1)_{m} \, (c_2)_n \, m! \, n!} \ x^m \, y^n, \label{3.2}
\end{align} 
where $(-t)_{m \, k}$ is given by \cite{bc4}
\begin{align}
	(-t)_{m \, k} &  = 	k^{m \, k} \, \prod_{i = 0}^{k - 1} \left(\frac{-t + i}{k}\right)_m.
\end{align}
In the subsequent sections, we explore various properties of these discrete functions \emph{viz.} regions of convergence, difference-differential equations, integral representations, difference-differential formulas, recursion relations and finite-infinite sums. For deeper understanding, we suggest to browse the papers and references therein \cite{bc1}-\cite{bc5}. The results obtained here are believed to be new. Also their particular cases bring out  the results for Appell function $F_2$ and Kamp\'e de F\'eriet hypergeometric functions. 

\section{The discrete Appell function $\mathcal{F}^{(1)}_2$}	
In this section, we establish the particular values of the function $\mathcal{F}^{(1)}_2$, its area of convergence, the difference equations and the integral representations. At the end of section, we derive the discrete version of the Humbert functions $\psi_1$ and $\psi_2$ as limiting cases of $\mathcal{F}^{(1)}_2$. We start our findings with the special values of $\mathcal{F}^{(1)}_2$. The particular values of $k_1$ and $k_2$, bring down the discrete Appell function $\mathcal{F}^{(1)}_2$ into some known classical functions. Especially when $k_1 = k_2 = 0$, we get
	  \begin{align}
	  	\mathcal{F}^{(1)}_2(a, b_1, b_2; c_1, c_2; t_1, t_2, 0, 0, x, y)  = F_2 (a, b_1, b_2; c_1, c_2; x, y).
	  \end{align}
      When $k_1 = 0$ and $k_2 = 1$, we have
      \begin{align}
     	& \mathcal{F}^{(1)}_2(a, b_1, b_2; c_1, c_2; t_1, t_2, 0, 1, x, y)
     	%\nonumber\\
     %	& = \sum_{m, n \geq 0} \frac{(a)_{m+n} \, (b_1)_m \, (b_2)_n \, (-1)^{n} \,  (-t_2)_n}{ (c_1)_{m} \, (c_2)_n \, m! \, n!} \ x^m \, y^n\nonumber\\
     %& 
     = F_{{0}:{1}, {1}} ^{{1}:{1}, {2}}\left(\begin{array}{ccc}
      	a: & b_1,  & b_2, -t_2\\
      	-: & c_1, & c_2 
      \end{array}; x, \ - y\right), 
      \end{align}
  where $F_{{0}:{1}, {1}} ^{{1}:{1}, {2}}$ is the Kamp\'e de F\'eriet  hypergeometric function;  defined in general by \cite{sk}
  \begin{align}
  	& F_{{l_2}:{l'_2}, {l''_2}} ^{{l_1}:{l'_1}, {l''_1}}\left(\begin{array}{ccc}
  			A: & B, & C\\
  			D: &E, &F 
  		\end{array}; x, \ y\right)\nonumber\\
  	& = %\displaystyle  
  	\sum_{m,n\geq 0} \ \frac{\prod_{i=1}^{l_1} (a_i)_{m+n} \, \prod_{i=1}^{l'_1} (b_i)_{m} \, \prod_{i=1}^{l''_1} (c_i)_{n}}{\prod_{i=1}^{l_2} (d_i)_{m+n}   \, \prod_{i=1}^{l'_2} (e_i)_{m}  \, \prod_{i=1}^{l''_2} (f_i)_{n}} \ \frac{x^m \, y^n}{m! \, n!},\label{c1eq71}
  \end{align}
where $A$ denote the sequence of complex numbers $a_1, \dots, a_{l_1}$. Similarly, for $k_1 = 1$ and $k_2 = 0$, we have
 \begin{align}
 	& \mathcal{F}^{(1)}_2(a, b_1, b_2; c_1, c_2; t_1, t_2, 1, 0, x, y)
 	%\nonumber\\
 %	& = \sum_{m, n \geq 0} \frac{(a)_{m+n} \, (b_1)_m \, (b_2)_n \, (-1)^{m} \,  (-t_1)_m}{ (c_1)_{m} \, (c_2)_n \, m! \, n!} \ x^m \, y^n\nonumber\\
 	%& 
 	= F_{{0}:{1}, {1}} ^{{1}:{2}, {1}}\left(\begin{array}{ccc}
 		a: & b_1, -t_1,  & b_2\\
 		-: & c_1, & c_2 
 	\end{array}; - x, \ y\right). 
 \end{align}
Further, for $k_1 = k_2 = 1$, we have
\begin{align}
	& \mathcal{F}^{(1)}_2(a, b_1, b_2; c_1, c_2; t_1, t_2, 1, 1, x, y)
	%\nonumber\\
%	& = \sum_{m, n \geq 0} \frac{(a)_{m+n} \, (b_1)_m \, (b_2)_n \,  (-1)^{m} \,  (-t_1)_m \, (-1)^{n} \,  (-t_2)_n}{ (c_1)_{m} \, (c_2)_n \, m! \, n!} \ x^m \, y^n\nonumber\\
	%& 
	= F_{{0}:{1}, {1}} ^{{1}:{2}, {2}}\left(\begin{array}{ccc}
		a: & b_1, -t_1,  & b_2, -t_2\\
		-: & c_1, & c_2 
	\end{array}; - x, \ - y\right). 
\end{align}
Now, we analyse the convergence of the discrete function $\mathcal{F}^{(1)}_2$. Let $\mathcal{A}_{m, n} x^m \, y^n$ be the general term of the discrete function $\mathcal{F}^{(1)}_2$. Then
  \begin{align}
  &	\left \vert \mathcal{A}_{m, n} x^m \, y^n \right \vert\nonumber\\
   & = \left \vert \frac{(a)_{m+n} \, (b_1)_m \, (b_2)_n \, (-t_1)_{mk} \, (-t_2)_{nk}}{ (c_1)_{m} \, (c_2)_n \, m! \, n!} \ x^m \, y^n \right \vert\nonumber\\
  	& < \left \vert \frac{\Gamma (c_1) \, \Gamma (c_2)}{\Gamma (a) \, \Gamma (b_1) \, \Gamma (b_2) \, \Gamma (-t_1) \, \Gamma (-t_2)} \right \vert \nonumber\\
  	& \quad \times \left \vert \frac{\Gamma (a + m + n) \, \Gamma (b_1 + m) \, \Gamma (b_2 + n) \, \Gamma (-t_1 + mk_1) \, \Gamma (-t_2 + nk_2)}{\Gamma (c_1 + m) \, \Gamma (c_2 + n) \Gamma (m + 1) \, \Gamma (n + 1)} \right \vert \, \vert x \vert^m \, \vert y\vert^n.
  \end{align}
The  Stirling formula, \cite{emo}
\begin{align}
	\lim_{n \to \infty} \Gamma (\mu + n) = \sqrt (2 \pi) \, n^{\mu + n - \frac{1}{2}} \, e^{-n},
\end{align}
for large values of $m$ and $n$, yield
\begin{align}
	&	\left \vert \mathcal{A}_{m, n} x^m \, y^n \right \vert\nonumber\\
	& < \left \vert \frac{2 \pi \, \Gamma (c_1) \, \Gamma (c_2)}{\Gamma (a) \, \Gamma (b_1) \, \Gamma (b_2) \, \Gamma (-t_1) \, \Gamma (-t_2)} \right \vert \nonumber\\
	& \quad \times \left \vert  (m + n)^{a - 1} \, m^{b_1 - c_1} \, n^{b_2 - c_2} \, (mk_1)^{mk_1 - t_1 - \frac{1}{2}} \, (nk_2)^{nk_2 - t_2 - \frac{1}{2}} e^{- (mk_1 + nk_2)} \right \vert \nonumber\\
	& \quad \times  \left\vert \frac{(m + n)!}{m ! \, n!} x^m \, y^n\right\vert.
\end{align}
Consider $N > \left \vert \frac{2 \pi \, \Gamma (c_1) \, \Gamma (c_2)}{\Gamma (a) \, \Gamma (b_1) \, \Gamma (b_2) \, \Gamma (-t_1) \, \Gamma (-t_2)} \right \vert$. Then
\begin{align}
		\left \vert \mathcal{A}_{m, n} x^m \, y^n \right \vert & < \frac{N \, (mk_1)^{mk_1 - t_1 - \frac{1}{2}} \, (nk_2)^{nk_2 - t_2 - \frac{1}{2}}}{ (m + n)^{1 - a} \, m^{c_1 - b_1} \, n^{c_2 - b_2} \, e^{ (mk_1 + nk_2)}} \, \left\vert \frac{(m + n)!}{m ! \, n!} x^m \, y^n\right\vert\nonumber\\
		& < \frac{N \, (mk_1)^{mk_1 - t_1 - \frac{1}{2}} \, (nk_2)^{nk_2 - t_2 - \frac{1}{2}}}{ (m + n)^{1 - a} \, m^{c_1 - b_1} \, n^{c_2 - b_2} \, e^{ (mk_1 + nk_2)}} (\vert x\vert + \vert y \vert)^{m + n}.
\end{align}
For $k_1, k_2 \in \mathbb{N}, t_1, t_2 \in \mathbb{C}$ and $\vert x\vert + \vert y \vert < 1$, $\left \vert \mathcal{A}_{m, n} x^m \, y^n \right \vert \to 0$ as $m, n \to \infty$. Hence the discrete function $\mathcal{F}^{(1)}_2$ converges absolutely. 
\subsection{Difference equations}
If $\Theta_{t}$ is considered as the discrete analogue of differential operator $t  \frac{d}{dt}$, defined by $\Theta_t : = t \, \rho_t \, \Delta_t$, where $\Delta_t f(t) = f(t + 1) - f(t)$ and $\rho_t \, f (t) = f(t - 1)$, then we have
\begin{align}
\Theta_t \, ((-1)^{nk} \, (-t)_{nk}) & = n\, k \, (-1)^{nk} \, (-t)_{nk}. 
\end{align} 
With this assumption, we get
\begin{align}
	& \Theta_{t_1} \left(\frac{1}{k_1} \Theta_{t_1} + c_1 - 1\right) \, \mathcal{F}^{(1)}_2\nonumber \\
	& = \sum_{m,n \geq 0} \frac{(a)_{m+n} \, (b_1)_m \, (b_2)_n \, (-1)^{m k_1} \, (-t_1)_{mk_1} \, (-1)^{n k_2} (-t_2)_{nk_2}}{ (c_1)_{m} \, (c_2)_n \, m! \, n!} \ x^m \, y^n \, mk_1 \, (c_1 + m - 1)\nonumber\\
%	& = k_1  \sum_{m \ge 1, n \ge 0} \frac{(a)_{m+n} \, (b_1)_m \, (b_2)_n \, (-1)^{(m + n) k} \, (-t_1)_{mk} \, (-t_2)_{nk}}{ (c_1)_{m - 1} \, (c_2)_n \, (m - 1)! \, n!} \ x^m \, y^n  \nonumber\\
	& = k_1 \, \sum_{m, n \ge 0} \frac{(a)_{m+n + 1} \, (b_1)_{m + 1} \, (b_2)_n \, (-1)^{(m + 1) k_1} \, (-t_1)_{(m + 1)k_1} \, (-1)^{n k_2} \, (-t_2)_{nk_2}}{ (c_1)_{m} \, (c_2)_n \, m! \, n!} \ x^{m + 1} \, y^n\nonumber\\
%	& = k \, \sum_{m, n \ge 0} (a + m + n) \, (b_1 + m) \, (-1)^k \, (-t_1)_k \, x \nonumber\\
%	& \quad \times \frac{(a)_{m+n} \, (b_1)_{m} \, (b_2)_n \, (-1)^{(m + n) k} \, (-t_1 + k)_{m k} \, (-t_2)_{nk}}{ (c_1)_{m} \, (c_2)_n \, m! \, n!} \, x^{m} \, y^n\nonumber\\
	& = k_1 \, \sum_{m, n \ge 0} (a + m + n) \, (b_1 + m) \, (-1)^{k_1} \, (-t_1)_{k_1} \, x \, \rho_{t_1}^{k_1} \nonumber\\
	& \quad \times \frac{(a)_{m+n} \, (b_1)_{m} \, (b_2)_n \, (-1)^{m k_1} \, (-t_1)_{m k_1} \, (-1)^{n k_2} \, (-t_2)_{nk_2}}{(c_1)_{m} \, (c_2)_n \, m! \, n!} \ x^{m} \, y^n\nonumber\\
	& = k_1 \, (-1)^{k_1} \, (-t_1)_{k_1} \, x \, \rho_{t_1}^{k_1} \,  \left(\frac{1}{k_1} \Theta_{t_1} + \frac{1}{k_2} \Theta_{t_2} + a\right)  \left(\frac{1}{k_1} \Theta_{t_1}  + b_1\right) \, \mathcal{F}^{(1)}_2.
\end{align}
Thus, we arrive at
\begin{align}
	\left[\Theta_{t_1} \left(\frac{1}{k_1} \Theta_{t_1} + c_1 - 1\right)  - k_1 \, (-1)^{k_1}  (-t_1)_{k_1} \, x \, \rho_{t_1}^{k_1}   \left(\frac{1}{k_1} \Theta_{t_1} + \frac{1}{k_2} \Theta_{t_2} + a\right) \left(\frac{1}{k_1} \Theta_{t_1}  + b_1\right)\right] \mathcal{F}^{(1)}_2=0.\label{1.15}
\end{align}	    
	    Similarly
\begin{align}
	\left[\Theta_{t_2} \left(\frac{1}{k_2} \Theta_{t_2} + c_2 - 1\right)  - k_2 \, (-1)^{k_2}  (-t_2)_{k_2} \, y \, \rho_{t_2}^{k_2}  \left(\frac{1}{k_1} \Theta_{t_1} + \frac{1}{k_2} \Theta_{t_2} + a\right)  \left(\frac{1}{k_2} \Theta_{t_2}  + b_2\right) \right] \mathcal{F}^{(1)}_2=0.\label{1.16}  
\end{align}
Equations \eqref{1.15} and \eqref{1.16} are difference equations satisfied by the discrete Appell function $\mathcal{F}^{(1)}_2$.

\subsection{Integral representations}
We now establish the integral representations for discrete Appell  function $\mathcal{F}^{(1)}_2$.
\begin{theorem}\label{t7}
	Let $a$, $b_1$, $b_2$, $c_1$, $c_2$, $t_1$, $t_2$ be complex numbers and $k_1, k_2 \in \mathbb{N}$. Then for $\vert x\vert  + \vert y\vert < 1$, the discrete function defined in \eqref{3.1} can be represented in the integral form as
	\begin{align}
	& \mathcal{F}^{(1)}_2(a, b_1, b_2; c_1, c_2; t_1, t_2, k_1, k_2 , x, y)\nonumber\\
	& =\Gamma \left(\begin{array}{c}
		c_1, c_2\\
		b_1, b_2, c_1 - b_1, c_2 - b_2
	\end{array}\right) \int_{0}^{1} \int_{0}^{1} u^{b_1 - 1} v^{b_2 - 1} (1-u)^{c_1 -  b_1 - 1} (1 - v)^{c_2 - b_2 - 1} \nonumber\\
& \quad \times F_{{0}:{0}; {0}} ^{{1}:{k_1}; {k_2}}\left(\begin{array}{ccc}
	a: & \frac{- t_1}{k_1}, \dots, \frac{- t_1 + k_1 - 1}{k_1} ; & \frac{- t_2}{k_2}, \dots, \frac{- t_2 + k_2 - 1}{k_2}\\
	-: & - ; & - 
\end{array}; (-k_1)^{k_1} \, u x,  (-k_2)^{k_2} \, v y\right) du \, dv. \label{1.18}
\end{align}
\end{theorem}	
\begin{proof}
Using 
\begin{align}
\frac{(b_1)_{m}}{(c_1)_{m}}
 & = \Gamma \left(\begin{array}{c}
 	c_1\\
 	b_1, c_1 - b_1
 \end{array}\right) \int_{0}^{1} u^{b_1 + m - 1} \, (1-u)^{c_1 - b_1 - 1} \, du, \label{58}
\end{align}
and 
\begin{align}
	\frac{(b_2)_{n}}{(c_2)_{n}}
	& = \Gamma \left(\begin{array}{c}
		c_2\\
		b_2, c_2 - b_2
	\end{array}\right) \int_{0}^{1} v^{b_2 + n - 1} \, (1-v)^{c_2 - b_2 - 1} \, dv, \label{58}
\end{align}
we can write
	\begin{align}
	& \mathcal{F}^{(1)}_2(a, b_1, b_2; c_1, c_2; t_1, t_2, k_1, k_2, x, y)\nonumber\\
	& =\Gamma \left(\begin{array}{c}
		c_1, c_2\\
		b_1, b_2, c_1 - b_1, c_2 - b_2
	\end{array}\right) \sum_{m, n \geq 0} \int_{0}^{1} \int_{0}^{1}  u^{b_1 + m - 1} \, (1-u)^{c_1 - b_1 - 1} \, v^{b_2 + n - 1} \nonumber\\
& \quad \times (1 - v)^{c_2 - b_2 - 1} \, \frac{(a)_{m + n} \, (-1)^{m k_1} \, (-t_1)_{mk_1} \, (-1)^{n k_2} \, (-t_2)_{nk_2}}{ m! \, n!} \ x^m \, y^n \, du \, \nonumber\\
& = \Gamma \left(\begin{array}{c}
	c_1, c_2\\
	b_1, b_2, c_1 - b_1, c_2 - b_2
\end{array}\right) \int_{0}^{1} \int_{0}^{1} u^{b_1 - 1} v^{b_2 - 1} (1-u)^{c_1 -  b_1 - 1} (1 - v)^{c_2 - b_2 - 1} \nonumber\\
& \quad \times \sum_{m, n \geq 0} \frac{(a)_{m + n}  \, (-1)^{m k_1} \, (-t_1)_{mk_1} \, (-1)^{ n k_2} \,  (-t_2)_{nk_2}  \, (u \, x)^m \, (uy)^n}{ m!, n! } \, du \, dv.
\end{align}
The implication of the result
\begin{align}
	(-t)_{mk} & = k^{mk} \, \left(\frac{-t}{k}\right)_m \, \left(\frac{-t + 1}{k}\right)_m \cdots \left(\frac{-t + k - 1}{k}\right)_m,
\end{align}
gives
\begin{align}
	& \mathcal{F}^{(1)}_2(a, b_1, b_2; c_1, c_2; t_1, t_2, k_1, k_2, x, y)\nonumber\\
	& =\Gamma \left(\begin{array}{c}
		c_1, c_2\\
		b_1, b_2, c_1 - b_1, c_2 - b_2
	\end{array}\right) \int_{0}^{1} \int_{0}^{1} u^{b_1 - 1} v^{b_2 - 1} (1-u)^{c_1 -  b_1 - 1} (1 - v)^{c_2 - b_2 - 1} \nonumber\\
	& \quad \times F_{{0}:{0}; {0}} ^{{1}:{k_1}; {k_2}}\left(\begin{array}{ccc}
		a: & \frac{- t_1}{k_1}, \dots, \frac{- t_1 + k_1 - 1}{k_1} ; & \frac{- t_2}{k_2}, \dots, \frac{- t_2 + k_2 - 1}{k_2}\\
		-: & - ; & - 
	\end{array}; (-k_1)^{k_1} \, u x,  (-k_2)^{k_2} \, v y\right) du \, dv. 
\end{align}
\end{proof}
Further, we can obtain integral forms for discrete Appell function $\mathcal{F}^{(1)}_2$ using the integral representation of gamma function, given by \cite{edr}
\begin{align}
	\Gamma (w) = \int_{0}^{\infty} e^{-u} \, u^{w - 1} \, du, \quad \Re(w) > 0.
\end{align}
 We list such integrals here for discrete Appell function $\mathcal{F}^{(1)}_2$, since proofs are straightforward we omit them. 
\begin{align}
& \mathcal{F}^{(1)}_2(a, b_1, b_2; c_1, c_2; t_1, t_2, k_1, k_2, x, y)\nonumber\\
& = \frac{1}{\Gamma (a)} \int_{0}^{\infty} e^{-u} \, u^{a - 1} \nonumber\\
& \quad \times F_{{0}:{1}; {1}} ^{{0}:{k_1 + 1}; {k_2 + 1}}\left(\begin{array}{ccc}
	- : & b_1, \frac{- t_1}{k_1}, \dots, \frac{- t_1 + k_1 - 1}{k_1} ; & b_2, \frac{- t_2}{k_2}, \dots, \frac{- t_2 + k_2 - 1}{k_2}\\
-	: & c_1 ; & c_2
\end{array}; (-k_1)^{k_1} \, u x,  (-k_2)^{k_2} \, u y\right) du;\label{3.9}\\
& = \frac{1}{\Gamma (b_1)} \int_{0}^{\infty} e^{-u} \, u^{b_1 - 1} \nonumber\\
& \quad \times F_{{0}:{1}; {1}} ^{{1}:{k_1}; {k_2 + 1}}\left(\begin{array}{ccc}
	a : & \frac{- t_1}{k_1}, \dots, \frac{- t_1 + k_1 - 1}{k_1} ; & b_2, \frac{- t_2}{k_2}, \dots, \frac{- t_2 + k_2 - 1}{k_2}\\
	-: & c_1 ; & c_2 
\end{array}; (-k_1)^{k_1} \, u x,  (-k_2)^{k_2} \, y\right) du;\\
& = \frac{1}{\Gamma (b_2)} \int_{0}^{\infty} e^{-v} \, v^{b_2 - 1} \nonumber\\
& \quad \times F_{{0}:{1}; {1}} ^{{1}:{k_1 + 1}; {k_2}}\left(\begin{array}{ccc}
	a : & b_1, \frac{- t_1}{k_1}, \dots, \frac{- t_1 + k_1 - 1}{k_1} ; & \frac{- t_2}{k_2}, \dots, \frac{- t_2 + k_2 - 1}{k_2}\\
	-: & c_1 ; & c_2 
\end{array}; (-k_1)^{k_1} \, x,  (-k_2)^{k_2} \, v y\right) dv;\\
& = \frac{1}{\Gamma (-t_1)} \int_{0}^{\infty} e^{-u} \, u^{-t_1 - 1} \nonumber\\
& \quad \times F_{{0}:{1}; {1}} ^{{1}:{1}; {k_2 + 1}}\left(\begin{array}{ccc}
a : & b_1 ; & b_2, \frac{- t_2}{k_2}, \dots, \frac{- t_2 + k_2 - 1}{k_2}\\
-: & c_1 ; & c_2 
\end{array}; (- u)^{k_1} \, x,  (-k_2)^{k_2} \, y\right) du;\\
& = \frac{1}{\Gamma (-t_2)} \int_{0}^{\infty} e^{-v} \, v^{-t_2 - 1} \nonumber\\
& \quad \times F_{{0}:{1}; {1}} ^{{1}:{k_1 + 1}; {1}}\left(\begin{array}{ccc}
a : & b_1, \frac{- t_1}{k_1}, \dots, \frac{- t_1 + k_1 - 1}{k_1} ; & b_2\\
-: & c_1 ; & c_2 
\end{array}; (-k_1)^{k_1} \, x,  (-v)^{k_2} \, y\right) dv.
\end{align}
We denote the first discrete version of Humbert functions $\psi_1$, $\psi_2$, \cite{ph, sk}, by $\psi^{(1)}_1$, $\psi^{(1)}_2$ and defined as:
\begin{align}
	&\psi^{(1)}_1 \left(a, b_1; c_1, c_2; t_1, t_2, k_1, k_2, x,  y\right) \nonumber\\
	& = \sum_{m,n \geq 0} \frac{(a)_{m+n} \, (b_1)_m  \, (-1)^{m k_1} \, (-t_1)_{mk_1} \, (-1)^{n k_2} \, (-t_2)_{nk_2}}{ (c_1)_{m} \, (c_2)_n \, m! \, n!} \ x^m \, y^n\nonumber
	\\[5pt]
&	\psi^{(1)}_2 \left( a; c_1, c_2; t_1, t_2, k_1, k_2, x,  y\right)\nonumber\\
 & = \sum_{m,n \geq 0} \frac{(a)_{m+n} \, (-1)^{m k_1} \, (-t_1)_{mk_1} \, (-1)^{n k_2} \, (-t_2)_{nk_2}}{ (c_1)_{m} \, (c_2)_n \, m! \, n!} \ x^m \, y^n.
\end{align}
One can verify that the limiting cases of discrete Appell function $\mathcal{F}^{(1)}_2$ give the discrete Humbert functions $\psi^{(1)}_1$ and $\psi^{(1)}_2$ as
\begin{align}
	& \lim_{\varepsilon \to 0}  \mathcal{F}^{(1)}_2 \left(a, b_1, \frac{1}{\varepsilon}; c_1, c_2; t_1, t_2, k_1, k_2, x, \varepsilon \, y\right) \nonumber\\
	& = \lim_{\varepsilon \to 0} \sum_{m,n \geq 0} \frac{(a)_{m+n} \, (b_1)_m \, \left(\frac{1}{\varepsilon}\right)_n \, (-1)^{m k_1} \, (-t_1)_{mk_1} \, (-1)^{n k_2} \, (-t_2)_{nk_2}}{ (c_1)_{m} \, (c_2)_n \, m! \, n!} \ x^m \, (\varepsilon \, y)^n\nonumber\\
	& = 	 \sum_{m,n \geq 0} \frac{(a)_{m+n} \, (b_1)_m \,  \, (-1)^{m k_1} \, (-t_1)_{mk_1} \, (-1)^{n k_2} \, (-t_2)_{nk_2}}{ (c_1)_{m} \, (c_2)_n \, m! \, n!} \ x^m \, y^n \, \lim_{\varepsilon \to 0} \varepsilon^n \left(\frac{1}{\varepsilon}\right)_n\nonumber\\
	& = 	 \sum_{m,n \geq 0} \frac{(a)_{m+n} \, (b_1)_m   \, (-1)^{m k_1} \, (-t_1)_{mk_1} \, (-1)^{n k_2} \, (-t_2)_{nk_2}}{ (c_1)_{m} \, (c_2)_n \, m! \, n!} \ x^m \, y^n\nonumber\\
	& = \psi^{(1)}_1 \left(a, b_1; c_1, c_2; t_1, t_2, k_1, k_2, x,  y\right).
\end{align}
Similarly
\begin{align}
	\lim_{\varepsilon \to 0}  \mathcal{F}^{(1)}_2 \left(a, \frac{1}{\varepsilon}, \frac{1}{\varepsilon} ; c_1, c_2; t_1, t_2, k_1, k_2, \varepsilon \, x, \varepsilon \, y\right) = \psi^{(1)}_2 \left( a; c_1, c_2; t_1, t_2, k_1, k_2, x,  y\right).
\end{align} 
\section{Differential and difference formulae for $\mathcal{F}^{(1)}_2$}
Let $\Delta_{t} \, f(t) = f(t + 1) - f(t)$ be the difference operator such that $\Delta_{t} [(-1)^k \, (-t)_k] = k \, (-1)^{k - 1} \, (-t)_{k - 1}$ and $\theta = x \frac{\partial}{\partial x}, \phi = y \frac{\partial}{\partial y}$ be the differential operators. Then,  we have the following 

\begin{theorem}
Following difference and differential formulae are satisfied by discrete Appell function $\mathcal{F}^{(1)}_2$:
\begin{align}
	&	(\Delta_{t_1})^r \mathcal{F}^{(1)}_2(a, b_1, b_2; c_1, c_2; t_1, t_2, 1, k_2, x, y) \nonumber\\
	& = \frac{(a)_r \, (b_1)_r \,  x^r }{(c_1)_r}  \mathcal{F}^{(1)}_2(a + r, b_1 + r, b_2; c_1 + r, c_2; t_1, t_2, 1, k_2, x, y);\label{4.1}\\
	& (\Delta_{t_2})^r \mathcal{F}^{(1)}_2(a, b_1, b_2; c_1, c_2; t_1, t_2, k_1, 1, x, y) \nonumber\\
	& = \frac{(a)_r \, (b_2)_r \,  y^r}{(c_2)_r}   \mathcal{F}^{(1)}_2(a + r, b_1, b_2 + r;c_1, c_2 + r; t_1, t_2, k_1, 1, x, y);\label{e32}\\
	%& (\Theta_{t_1})^r \mathcal{F}^{(1)}_2(a, b_1, b_2; c_1, c_2; t_1, t_2, k, x, y) \nonumber\\
	%& = \frac{(-1)^{rk} \, (a)_r \, (b_1)_r \, (-t_1)_{rk} \, x^r \, k^r}{(c_1)_r} \, \mathcal{F}^{(1)}_2(a + r, b_1 + r, b_2; c_1 + r, c_2; t_1 - rk, t_2, k, x, y);\\
	%& (\Theta_{t_2})^r \mathcal{F}^{(1)}_2(a, b_1, b_2; c_1, c_2; t_1, t_2, k, x, y) \nonumber\\
	%& = \frac{(-1)^{rk} \, (a)_r \, (b_2)_r \, (-t_2)_{rk} \, y^r \, k^r}{(c_2)_r} \, \mathcal{F}^{(1)}_2(a + r, b_1, b_2 + r; c_1, c_2 + r; t_1, t_2 - rk, k, x, y);\\
	& (\theta)^r \mathcal{F}^{(1)}_2(a, b_1, b_2; c_1, c_2; t_1, t_2, k_1, k_2, x, y) \nonumber\\
	& = \frac{(-1)^{rk_1} \, (a)_r \, (b_1)_r \, (-t_1)_{rk_1} \, x^r}{(c_1)_r} \nonumber\\
	& \quad \times  \mathcal{F}^{(1)}_2(a + r, b_1 + r, b_2; c_1 + r, c_2; t_1 - rk_1, t_2, k_1, k_2, x, y);\\
	& (\phi)^r \mathcal{F}^{(1)}_2(a, b_1, b_2; c_1, c_2; t_1, t_2, k_1, k_2, x, y) \nonumber\\
	& = \frac{(-1)^{rk_2} \, (a)_r \, (b_2)_r \, (-t_2)_{rk_2} \, y^r}{(c_2)_r} \nonumber\\
	& \quad \times  \mathcal{F}^{(1)}_2(a + r, b_1, b_2 + r; c_1, c_2 + r; t_1, t_2 - rk_2, k_1, k_2, x, y);\label{e34}\\
	& \left(\frac{\partial}{\partial x}\right)^r \left[x^{b_1 + r - 1} \mathcal{F}^{(1)}_2(a, b_1, b_2; c_1, c_2; t_1, t_2, k_1, k_2, x, y)\right]\nonumber\\
	& = x^{b_1 - 1} \, (b_1)_r \mathcal{F}^{(1)}_2(a, b_1 + r, b_2; c_1, c_2; t_1, t_2, k_1, k_2, x, y);\label{4.14}\\
	& \left(\frac{\partial}{\partial y}\right)^r [y^{b_2 + r - 1} \mathcal{F}^{(1)}_2(a, b_1, b_2; c_1, c_2; t_1, t_2,  k_1, k_2, x, y)]\nonumber\\
	& = y^{b_2 - 1} \, (b_2)_r \mathcal{F}^{(1)}_2 (a, b_1, b_2 + r; c_1, c_2; t_1, t_2,  k_1, k_2, x, y);\label{e3.6}\\
	& \left(\frac{\partial}{\partial x}\right)^r [x^{a + r - 1} \mathcal{F}^{(1)}_2(a, b_1, b_2; c_1, c_2; t_1, t_2, k_1, k_2, x, xy)]\nonumber\\
	& = x^{a - 1} \, (a)_r \mathcal{F}^{(1)}_2 (a + r, b_1, b_2; c_1, c_2; t_1, t_2,  k_1, k_2, x, xy);\\
	& \left(\frac{\partial}{\partial y}\right)^r [y^{a + r - 1} \mathcal{F}^{(1)}_2 (a, b_1, b_2; c_1, c_2; t_1, t_2,  k_1, k_2, xy, y)]\nonumber\\
	& = y^{a - 1} \, (a)_r \mathcal{F}^{(1)}_2(a + r, b_1, b_2; c_1, c_2; t_1, t_2,  k_1, k_2, xy, y);\\
	& \left(\frac{\partial}{\partial x}\right)^r [x^{c_1 - 1} \mathcal{F}^{(1)}_2(a, b_1, b_2; c_1, c_2; t_1, t_2,  k_1, k_2, x, y)]\nonumber\\
	& = (-1)^r \, (1 - c_1)_r \, x^{c_1 - r - 1} \mathcal{F}^{(1)}_2(a, b_1, b_2; c_1 - r, c_2; t_1, t_2,  k_1, k_2, x, y);\\
	& \left(\frac{\partial}{\partial y}\right)^r [y^{c_2 - 1} \mathcal{F}^{(1)}_2(a, b_1, b_2; c_1, c_2; t_1, t_2,  k_1, k_2, x, y)]\nonumber\\
	& = (-1)^r \, y^{c_2 - r - 1} \, (1 - c_2)_r \mathcal{F}^{(1)}_2(a, b_1, b_2; c_1, c_2 - r; t_1, t_2,  k_1, k_2, x, y).\label{e3.10}
\end{align}
\end{theorem}
\begin{proof}
 The proofs of these differential formulae are very elementary and hence we give only the proof of \eqref{4.1}.  The action of difference operator $\Delta_{t_1}$ on the discrete function $\mathcal{F}^{(1)}_2$ yields
 \begin{align}
 &(\Delta_{t_1}) \mathcal{F}^{(1)}_2(a, b_1, b_2; c_1, c_2; t_1, t_2, 1, k_2, x, y)\nonumber\\
& = \sum_{m,n \geq 0} \frac{(a)_{m+n} \, (b_1)_m \, (b_2)_n \, (-1)^{m - 1} \, m \, (-t_1)_{m - 1} \, (-1)^{nk_2} (-t_2)_{nk_2}}{ (c_1)_{m} \, (c_2)_n \, m! \, n!} \ x^m \, y^n\nonumber\\
%& =  \sum_{m \ge 1, n \geq 0} \frac{(a)_{m+n} \, (b_1)_m \, (b_2)_n \, (-1)^{(m + n) k - 1} \, (-t_1)_{mk - 1} \, (-t_2)_{nk}}{(c_1)_{m} \, (c_2)_n \, (m - 1)! \, n!} \ x^m \, y^n\nonumber\\
& = \sum_{m, n \geq 0} \frac{(a)_{m+n + 1} \, (b_1)_{m + 1} \, (b_2)_n \, (-1)^{m}  \, (-t_1)_{m} \, (-1)^{n k_2} (-t_2)_{nk_2}}{(c_1)_{m + 1} \, (c_2)_n \, m! \, n!} \ x^{m + 1} \, y^n\nonumber\\
%& =   \frac{a \, b_1 \, x}{c_1}\nonumber\\
%& \quad \times \sum_{m,n \geq 0} \frac{(a + 1)_{m+n} \, (b_1 + 1)_m \, (b_2)_n \, (-1)^{m} \, (-t_1 )_{mk} \, (-t_2)_{nk}}{(c_1 + 1)_{m} \, (c_2)_n \, m! \, n!} \ x^m \, y^n\nonumber\\
& =  \frac{a \, b_1 \, x}{c} \mathcal{F}^{(1)}_2 (a + 1, b_1 + 1, b_2; c_1 + 1, c_2; t_1, t_2, 1, k_2, x, y). 
 \end{align}
The result is true for $r = 1$. We can assume that this is also true for $r = p$ and using this it can be proved for $r = p + 1$. Thus, we get the required result. Similarly, we can show that the formula \eqref{e32}-\eqref{e34} hold true. To prove \eqref{4.14}, we start with 
\begin{align}
	& \frac{\partial}{\partial x} \left[x^{b_1} \mathcal{F}^{(1)}_2(a, b_1, b_2; c_1, c_2; t_1, t_2, k, x, y)\right]\nonumber\\
	& = \sum_{m, n \geq 0} \frac{(a)_{m+n} \, (b_1)_m \, (b_2)_n \, (-1)^{(m + n) k} \, (-t_1)_{mk} \, (-t_2)_{nk}}{ (c_1)_{m} \, (c_2)_n \, m! \, n!} \, \frac{\partial}{\partial x} x^{b_1 + m} \, y^n\nonumber\\
	& = \sum_{m, n \geq 0} \frac{(a)_{m+n} \, (b_1)_m \, (b_2)_n \, (-1)^{(m + n) k} \, (-t_1)_{mk} \, (-t_2)_{nk}}{ (c_1)_{m} \, (c_2)_n \, m! \, n!} \, (b_1 + m) \, x^{b_1 + m - 1} \, y^n\nonumber\\
	& = x^{b_1 - 1} \, b_1 \sum_{m, n \geq 0} \frac{(a)_{m+n} \, (b_1 + 1)_m \, (b_2)_n \, (-1)^{(m + n) k} \, (-t_1)_{mk} \, (-t_2)_{nk}}{ (c_1)_{m} \, (c_2)_n \, m! \, n!} \,  x^{m} \, y^n\nonumber\\
	& = x^{b_1 - 1} \, b_1 \mathcal{F}^{(1)}_2(a, b_1 + 1, b_2; c_1, c_2; t_1, t_2, k, x, y). 
\end{align}
So \eqref{4.14} is true for $r = 1$. Consider the result to be true for $r = p$ leads the validity of results for $r = p + 1$. Hence, we can conclude that \eqref{4.14} is true for each natural number $r$. Similarly, the results \eqref{e3.6}-\eqref{e3.10} can be proved.
\end{proof}
\section{Finite and infinite summation formulas}
In this section, we establish some finite and infinite sums in terms of discrete Appell function $\mathcal{F}^{(1)}_2$ in the following theorem.
\begin{theorem}
Following summation formulas hold:
	\begin{align}
& \mathcal{F}^{(1)}_2(a, b_1 + r, b_2; c_1, c_2; t_1, t_2, k_1, k_2, x, y)\nonumber\\
& = \sum_{s = 0}^{r} {r \choose s} \frac{(a)_s \, (-1)^{sk_1} \, (-t_1)_{sk_1}}{(c_1)_s} \, x^s \nonumber\\
& \quad \times  \mathcal{F}^{(1)}_2(a + s, b_1 + s, b_2; c_1 + s, c_2; t_1 - sk_1, t_2, k_1, k_2, x, y);\label{e5.1}\\
& \mathcal{F}^{(1)}_2(a, b_1, b_2 + r; c_1, c_2; t_1, t_2, k_1, k_2, x, y)\nonumber\\
& = \sum_{s = 0}^{r} {r \choose s} \frac{(a)_s \, (-1)^{sk_2} \, (-t_2)_{sk_2}}{(c_2)_s} \, y^s \nonumber\\
& \quad \times  \mathcal{F}^{(1)}_1(a + s, b_1, b_2 + s; c_1, c_2 + s; t_1, t_2 - sk_2, k_1, k_2, x, y);\label{e5.2}\\
&\sum_{r = 0}^{\infty} \frac{(a)_r}{r !} \, z^r \, \mathcal{F}^{(1)}_2 (a + r, b_1, b_2; c_1, c_2 ; t_1, t_2, k_1, k_2, x, y)\nonumber\\
& = (1 - z)^{-a} \, \mathcal{F}^{(1)}_2 \left(a, b_1, b_2; c_1, c_2; t_1, t_2, k_1, k_2, \frac{x}{1 - z}, \frac{y}{1 - z}\right);\label{e43}\\
&\sum_{r = 0}^{\infty} \frac{(b_1)_r}{r !} \, z^r \, \mathcal{F}^{(1)}_2 (a, b_1 + r, b_2; c_1, c_2 ; t_1, t_2, k_1, k_2, x, y)\nonumber\\
& = (1 - z)^{-b_1} \, \mathcal{F}^{(1)}_2 \left(a, b_1, b_2; c_1, c_2; t_1, t_2, k_1, k_2, \frac{x}{1 - z}, y\right);\\
&\sum_{r = 0}^{\infty} \frac{(b_2)_r}{r !} \, z^r \, \mathcal{F}^{(1)}_2 (a, b_1, b_2 + r; c_1, c_2 ; t_1, t_2, k_1, k_2, x, y)\nonumber\\
& = (1 - z)^{-b_2} \, \mathcal{F}^{(1)}_2 \left(a, b_1, b_2; c_1, c_2; t_1, t_2, k_1, k_2, x, \frac{y}{1 - z}\right).
 	\end{align}
\end{theorem}  
\begin{proof}
To prove \eqref{e5.1}, we start by applying the Leibnitz rule and get
\begin{align}
	& \left(\frac{\partial}{\partial x}\right)^r \left[x^{b_1 + r - 1} \, \mathcal{F}^{(1)}_2(a, b_1, b_2; c_1, c_2; t_1, t_2, k_1, k_2, x, y)\right]\nonumber\\
	& = \sum_{s = 0}^{r} {r \choose s} \left[\left(\frac{\partial}{\partial x}\right)^{r - s} \, x^{b_1 + r - 1}\right] \, \left[\left(\frac{\partial}{\partial x}\right)^s  \, \mathcal{F}^{(1)}_2(a, b_1, b_2; c_1, c_2; t_1, t_2, k_1, k_2, x, y)\right]\nonumber\\
	& = \sum_{s = 0}^{r} {r \choose s} (-1)^{r - s} \, (1 - b_1 + s)_{r - s} \, x^{b_1 + s - 1} \, \frac{(a)_s \, (b_1)_s \, (-1)^{sk_1} \, (-t_1)_{sk_1}}{(c_1)_s}\nonumber\\
	& \quad \times \mathcal{F}^{(1)}_2(a + s, b_1 + s, b_2; c_1 + s, c_2; t_1 - sk_1, t_2, k_1, k_2, x, y) \nonumber\\
	& = \sum_{s = 0}^{r} {r \choose s} \, (-1)^{r - s} \frac{(-1)^s \, (1 - b_1 - r)_r}{(b_1)_s} \, x^{b_1 + s - 1} \, \frac{(a)_s \, (b_1)_s \, (-1)^{sk_1} \, (-t_1)_{sk_1}}{(c_1)_s}\nonumber\\
	& \quad \times \mathcal{F}^{(1)}_2(a + s, b_1 + s, b_2; c_1 + s, c_2; t_1 - sk_1, t_2, k_1, k_2, x, y) \nonumber\\
	& = (b_1)_r \, \sum_{s = 0}^{r} {r \choose s} \,  x^{b_1 + s - 1} \, \frac{(a)_s  \, (-1)^{sk_1} \, (-t_1)_{sk_1}}{(c_1)_s}\nonumber\\
	& \quad \times \mathcal{F}^{(1)}_2(a + s, b_1 + s, b_2; c_1 + s, c_2; t_1 - sk_1, t_2, k_1, k_2, x, y).
\end{align}
Taking into account the equation \eqref{4.14}, we have
\begin{align}
&	x^{b_1 - 1} \, (b_1)_r \, \mathcal{F}^{(1)}_2 (a, b_1 + r, b_2; c_1, c_2; t_1, t_2, k_1, k_2, x, y)\nonumber\\
	& = (b_1)_r \, \sum_{s = 0}^{r} {r \choose s} \,  x^{b_1 + s - 1} \, \frac{(a)_s  \, (-1)^{sk_1} \, (-t_1)_{sk_1}}{(c_1)_s}\nonumber\\
	& \quad \times \mathcal{F}^{(1)}_2(a + s, b_1 + s, b_2; c_1 + s, c_2; t_1 - sk_1, t_2, k_1, k_2, x, y).
\end{align}
This implies
	\begin{align}
	& \mathcal{F}^{(1)}_2(a, b_1 + r, b_2; c_1, c_2; t_1, t_2, k_1, k_2, x, y)\nonumber\\
	& = \sum_{s = 0}^{r} {r \choose s} \frac{(a)_s \, (-1)^{sk_1} \, (-t_1)_{sk_1}}{(c_1)_s} \, x^s \nonumber\\
	& \quad \times  \mathcal{F}^{(1)}_2(a + s, b_1 + s, b_2; c_1 + s, c_2; t_1 - sk_1, t_2, k_1, k_2, x, y).
\end{align}
This completes the proof. A similar procedure can be used to get equation \eqref{e5.2}. To prove \eqref{e43}, we start with the right hand side of the equation and thus get
\begin{align}
	&(1 - z)^{-a} \, \mathcal{F}^{(1)}_2 \left(a, b_1, b_2; c_1, c_2; t_1, t_2, k_1, k_2, \frac{x}{1 - z}, \frac{y}{1 - z}\right) \nonumber\\
	& = \sum_{m,n \geq 0} (1 - z)^{- (a + m + n)} \, \frac{(a)_{m+n} \, (b_1)_m \, (b_2)_n \, (-1)^{m k_1} \, (-t_1)_{mk_1} \, (-1)^{n k_2} \, (-t_2)_{nk_2}}{ (c_1)_{m} \, (c_2)_n \, m! \, n!} \ x^m \, y^n.
\end{align}
Using the binomial theorem $(1 - z)^{- (a + m + n)} = \sum_{r = 0}^{\infty} \frac{(a + m + n)_r}{r !} \, z^r$ and identity $(a)_{m + n + r} = (a)_{m + n} \, (a + m + n)_r = (a)_r \, (a + r)_{m + n}$, we have
\begin{align}
	&(1 - z)^{-a} \, \mathcal{F}^{(1)}_2 \left(a, b_1, b_2; c_1, c_2; t_1, t_2, k_1, k_2, \frac{x}{1 - z}, \frac{y}{1 - z}\right) \nonumber\\
	& = \sum_{r = 0}^{\infty} \frac{(a)_r}{r !} \, z^r \, \mathcal{F}^{(1)}_2 (a + r, b_1, b_2; c_1, c_2 ; t_1, t_2, k, x, y).
\end{align}
It completes the proof of \eqref{e43}. A similar procedure can be used to prove remaining infinite summation formulas. 
\end{proof}
\section{Recursion formulae}
We have the following
\begin{theorem}
The following recursion formulas hold for the discrete Appell function $\mathcal{F}^{(1)}_2$:
\begin{align}
	& \mathcal{F}^{(1)}_2 (a + s, b_1, b_2; c_1, c_2 ; t_1, t_2, k_1, k_2, x, y) \nonumber\\
	& = \mathcal{F}^{(1)}_2 (a, b_1, b_2; c_1, c_2; t_1, t_2, k, x, y)\nonumber\\
	& \quad  + \frac{(-1)^k \, (-t_1)_k \, b_1 \, x}{c_1}  \sum_{r = 1}^{s} \mathcal{F}^{(1)}_2 (a + r, b_1 + 1, b_2; c_1 + 1, c_2; t_1 - k, t_2, k, x, y)\nonumber\\
	& \quad  + \frac{(-1)^k \, (-t_2)_k \, b_2 \, y}{c_2} \sum_{r = 1}^{s} \mathcal{F}^{(1)}_2 (a + r, b_1, b_2 + 1; c_1, c_2 + 1; t_1, t_2 - k, k, x, y);\label{e6.1}\\
	& \mathcal{F}^{(1)}_2 (a - s, b_1, b_2; c_1, c_2 ; t_1, t_2, k, x, y) \nonumber\\
	& = \mathcal{F}^{(1)}_2 (a, b_1, b_2; c_1, c_2 ; t_1, t_2, k, x, y)\nonumber\\
	& \quad  - \frac{(-1)^k \, (-t_1)_k \, b_1 \, x}{c_1} \sum_{r = 0}^{s - 1} \mathcal{F}^{(1)}_2 (a - r, b_1 + 1, b_2; c_1 + 1, c_2; t_1 - k, t_2, k, x, y)\nonumber\\
	& \quad  - \frac{(-1)^k \, (-t_2)_k \, b_2 \, y}{c_2} \sum_{r = 0}^{s - 1} \mathcal{F}^{(1)}_2 (a - r, b_1, b_2 + 1; c_1, c_2 + 1; t_1, t_2 - k, k, x, y);\\
	& \mathcal{F}^{(1)}_2 (a, b_1 + s, b_2; c_1, c_2 ; t_1, t_2, k, x, y) \nonumber\\
	& = \mathcal{F}^{(1)}_2 (a, b_1, b_2; c_1, c_2 ; t_1, t_2, k, x, y)\nonumber\\
	& \quad  + \frac{(-1)^k \, (-t_1)_k \, a \, x}{c_1}  \sum_{r = 1}^{s} \mathcal{F}^{(1)}_2 (a + 1, b_1 + r, b_2; c_1 + 1, c_2; t_1 - k, t_2, k, x, y);\\
	& \mathcal{F}^{(1)}_2 (a, b_1 - s, b_2; c_1, c_2 ; t_1, t_2, k, x, y) \nonumber\\
	& = \mathcal{F}^{(1)}_2 (a, b_1, b_2; c_1, c_2; t_1, t_2, k, x, y)\nonumber\\
	& \quad  - \frac{(-1)^k \, (-t_1)_k \, a \, x}{c_1}  \sum_{r = 0}^{s - 1} \mathcal{F}^{(1)}_2 (a + 1, b_1 - r, b_2; c_1 + 1, c_2; t_1 - k, t_2, k, x, y);\\
	& \mathcal{F}^{(1)}_2 (a, b_1, b_2; c_1 - s, c_2; t_1, t_2, k, x, y) \nonumber\\
	& = \mathcal{F}^{(1)}_2 (a, b_1, b_2; c_1, c_2 ; t_1, t_2, k, x, y)\nonumber\\
	& \quad  + (-1)^k \, (-t_1)_k \, a \, b_1 \, x  \sum_{r = 1}^{s} \frac{\mathcal{F}^{(1)}_2 (a + 1, b_1 + 1, b_2; c_1 + 2 - r, c_2; t_1 - k, t_2, k, x, y)}{(c_1 - r) \, (c_1 - r + 1)}\nonumber\\
	& \quad + (-1)^k \, (-t_2)_k \, a \, b_2 \, y  \sum_{r = 1}^{s} \frac{\mathcal{F}^{(1)}_2 (a + 1, b_1, b_2 + 1; c_1 + 2 - r, c_2; t_1, t_2 - k, k, x, y)}{(c_1 - r) \, (c_1 - r + 1)}.
\end{align}
\end{theorem}
\begin{proof}
To prove the formula \eqref{e6.1}, we begin with
\begin{align}
 &\mathcal{F}^{(1)}_2 (a, b_1, b_2; c_1, c_2 ; t_1, t_2, k, x, y) \nonumber\\
 & \quad + \frac{(-1)^k \, (-t_1)_k \, b_1 \, x}{c_1}   \mathcal{F}^{(1)}_2 (a + 1, b_1 + 1, b_2; c_1 + 1, c_2; t_1 - k, t_2, k, x, y) \nonumber\\
 & \quad  + \frac{(-1)^k \, (-t_2)_k \, b_2 \, y}{c_2}  \mathcal{F}^{(1)}_2 (a + 1, b_1, b_2 + 1; c_1, c_2 + 1; t_1, t_2 - k, k, x, y)\nonumber\\
& = \sum_{m,n \geq 0}  \, \frac{(a)_{m+n} \, (b_1)_m \, (b_2)_n \, (-1)^{(m + n) k} \, (-t_1)_{mk} \, (-t_2)_{nk}}{ (c_1)_{m} \, (c_2)_n \, m! \, n!} \ x^m \, y^n + \frac{(-1)^k \, (-t_1)_k \, b_1 \, x}{c_1} \nonumber\\
& \quad \times \sum_{m,n \geq 0}  \, \frac{(a + 1)_{m+n} \, (b_1 + 1)_m \, (b_2)_n \, (-1)^{(m + n) k} \, (-t_1 + k)_{mk} \, (-t_2)_{nk}}{ (c_1 + 1)_{m} \, (c_2)_n \, m! \, n!} \ x^m \, y^n + \frac{(-1)^k \, (-t_2)_k \, b_2 \, y}{c_2} \nonumber\\
& \quad \times \sum_{m, n \geq 0}  \, \frac{(a + 1)_{m+n} \, (b_1)_m \, (b_2 + 1)_n \, (-1)^{(m + n) k} \, (-t_1)_{mk} \, (-t_2 + k)_{nk}}{(c_1)_m \, (c_2 + 1)_{n} \, m! \, n!} \ x^m \, y^n\nonumber\\
& = \sum_{m,n \geq 0}  \, \frac{(a)_{m+n} \, (b_1)_m \, (b_2)_n \, (-1)^{(m + n) k} \, (-t_1)_{mk} \, (-t_2)_{nk}}{ (c_1)_{m} \, (c_2)_n \, m! \, n!} \ x^m \, y^n \nonumber\\
& \quad + \sum_{m, n \geq 0} \frac{m}{a} \, \frac{(a)_{m+n} \, (b_1)_m \, (b_2)_n \, (-1)^{(m + n) k} \, (-t_1)_{mk} \, (-t_2)_{nk}}{(c_1)_{m} \, (c_2)_n \, m! \, n!} \ x^m \, y^n \nonumber\\
& \quad + \sum_{m, n \geq 0}  \frac{n}{a} \, \frac{(a)_{m+n} \, (b_1)_m \, (b_2)_n \, (-1)^{(m + n) k} \, (-t_1)_{mk} \, (-t_2)_{nk}}{ (c_1)_{m} \, (c_2)_n \, m! \, n!} \ x^m \, y^n\nonumber\\
& =  \sum_{m,n \geq 0}  \frac{a + m + n}{a} \, \frac{(a)_{m+n} \, (b_1)_m \, (b_2)_n \, (-1)^{(m + n) k} \, (-t_1)_{mk} \, (-t_2)_{nk}}{ (c_1)_{m} \, (c_2)_n \, m! \, n!} \ x^m \, y^n\nonumber\\
& = \mathcal{F}^{(1)}_2 (a + 1, b_1, b_2; c_1, c_2; t_1, t_2, k, x, y). 
\end{align}
This shows the result is true for $s = 1$. Assuming the result to be true for $s = p$ leads to the validity of results for $s = p + 1$. 
 Hence the formula is true for each natural number.  The other formulae can be proved in the same manner. 
\end{proof}
Besides, we have several other difference and differential recursion formulae. We start our findings with the differential recursion formulae first. To obtain, we list simple differential relations as
\begin{align}
&	a \, \mathcal{F}^{(1)}_2 (a + 1)  = (a + \theta + \phi) \, \mathcal{F}^{(1)}_2;\\
& (a + \theta + \phi - 1) \, \mathcal{F}^{(1)}_2 (a - 1)  =	(a - 1) \, \mathcal{F}^{(1)}_2;\\
&	b_1 \, \mathcal{F}^{(1)}_2 (b_1 + 1)  = (b_1 + \theta) \, \mathcal{F}^{(1)}_2;\\
&(b_1 + \theta - 1) \, \mathcal{F}^{(1)}_2 (b_1 - 1)  =	(b_1 - 1) \, \mathcal{F}^{(1)}_2;\\
&	b_2 \, \mathcal{F}^{(1)}_2 (b_2 + 1)  = (b_2 + \phi) \, \mathcal{F}^{(1)}_2;\\
& (b_2  + \phi - 1) \, \mathcal{F}^{(1)}_2 (b_2 - 1)  =	(b_2 - 1) \, \mathcal{F}^{(1)}_2;\\
&	(c_1 - 1) \, \mathcal{F}^{(1)}_2 (c_1 - 1)  = (c_1 + \theta - 1) \, \mathcal{F}^{(1)}_2;\\
& (c_1 + \theta) \, \mathcal{F}^{(1)}_2 (c_1 + 1)  =	c_1 \, \mathcal{F}^{(1)}_2;\\
&	(c_2 - 1) \, \mathcal{F}^{(1)}_2 (c_2 - 1)  = (c_2 + \phi - 1) \, \mathcal{F}^{(1)}_2;\\
& (c_2 + \phi) \, \mathcal{F}^{(1)}_2 (c_2 + 1)  =	c_2 \, \mathcal{F}^{(1)}_2.
\end{align}
On combining any two of the above relations, we get first order or second order differential recursion relations. We produce here a list of such relations.
\begin{align}
& a \, (a - 1) \, \mathcal{F}^{(1)}_2 (a + 1) - (a + \theta + \phi) \, (a + \theta + \phi - 1) \mathcal{F}^{(1)}_2 (a - 1) = 0;\\
& a \, (b_1 - 1) \, \mathcal{F}^{(1)}_2 (a + 1) - (a + \theta + \phi) \, (b_1 + \theta  - 1) \mathcal{F}^{(1)}_2 (b_1 - 1) = 0;\\
& a \, (b_2 - 1) \, \mathcal{F}^{(1)}_2 (a + 1) - (a + \theta + \phi) \, (b_2 + \phi - 1) \mathcal{F}^{(1)}_2 (b_2 - 1) = 0;\\
& a \, c_1 \, \mathcal{F}^{(1)}_2 (a + 1) - (a + \theta + \phi) \, (c_1 + \theta) \mathcal{F}^{(1)}_2 (c_1 + 1) = 0;\\
& a \, c_2 \, \mathcal{F}^{(1)}_2 (a + 1) - (a + \theta + \phi) \, (c_2 + \phi) \mathcal{F}^{(1)}_2 (c_2 + 1) = 0;\\
& a \, (b_1 + \theta) \, \mathcal{F}^{(1)}_2 (a + 1) - b_1 \, (a + \theta + \phi) \, \mathcal{F}^{(1)}_2 (b_1 + 1) = 0;\\
		& a \, (b_2 + \phi) \, \mathcal{F}^{(1)}_2 (a + 1) - b_2 \, (a + \theta + \phi) \, \mathcal{F}^{(1)}_2 (b_2 + 1) = 0;\\
	& a \, (c_1 + \theta - 1) \, \mathcal{F}^{(1)}_2 (a + 1) - (c_1 - 1) \, (a + \theta + \phi) \, \mathcal{F}^{(1)}_2 (c_1 - 1) = 0;\\
	& a \, (c_2 + \phi - 1) \, \mathcal{F}^{(1)}_2 (a + 1) - (c_2 - 1) \, (a + \theta + \phi) \, \mathcal{F}^{(1)}_2 (c_2 - 1) = 0;\\
&(a + \theta + \phi - 1) \, (b_1 + \theta) \, \mathcal{F}^{(1)}_2 (a - 1) - b_1 \, (a - 1)  \, \mathcal{F}^{(1)}_2 (b_1 + 1) = 0;\\
 &	(a + \theta + \phi - 1) \, (b_2 + \phi) \, \mathcal{F}^{(1)}_2 (a - 1) - b_2 \, (a - 1)  \, \mathcal{F}^{(1)}_2 (b_2 + 1) = 0;\\
 &	(a + \theta + \phi - 1) \, (c_1 + \theta - 1) \, \mathcal{F}^{(1)}_2 (a - 1) - (c_1 - 1) \, (a - 1)  \, \mathcal{F}^{(1)}_2 (c_1 - 1) = 0;\\
 & (a + \theta + \phi - 1) \, (c_2 + \phi - 1) \, \mathcal{F}^{(1)}_2 (a - 1) - (c_2 - 1) \, (a - 1)  \, \mathcal{F}^{(1)}_2 (c_2 - 1) = 0;\\
 & (b_1 - 1) \,	(a + \theta + \phi - 1) \,  \mathcal{F}^{(1)}_2 (a - 1) -  (a - 1)  \, (b_1 + \theta - 1) \mathcal{F}^{(1)}_2 (b_1 - 1) = 0;\\
 & (b_2 - 1) \,	(a + \theta + \phi - 1) \,  \mathcal{F}^{(1)}_2 (a - 1) -  (a - 1)  \, (b_2 + \phi - 1) \mathcal{F}^{(1)}_2 (b_2 - 1) = 0;\\
 & c_1 \,	(a + \theta + \phi - 1) \,  \mathcal{F}^{(1)}_2 (a - 1) -  (a - 1)  \, (c_1 + \theta) \mathcal{F}^{(1)}_2 (c_1 + 1) = 0;\\
 & c_2 \,	(a + \theta + \phi - 1) \,  \mathcal{F}^{(1)}_2 (a - 1) -  (a - 1)  \, (c_2 + \phi) \mathcal{F}^{(1)}_2 (c_2 + 1) = 0;\\
 & b_1 \,	(b_1 - 1) \,  \mathcal{F}^{(1)}_2 (b_1 + 1) -  (b_1 + \theta)  \, (b_1 + \theta - 1) \mathcal{F}^{(1)}_2 (b_1 - 1) = 0;\\
 & b_1 \,	(b_2 + \phi) \,  \mathcal{F}^{(1)}_2 (b_1 + 1) -  b_2  \, (b_1 + \theta ) \mathcal{F}^{(1)}_2 (b_2 + 1) = 0;\\
  & b_1 \,	(b_2 - 1) \,  \mathcal{F}^{(1)}_2 (b_1 + 1) -  (b_1 + \theta)  \, (b_2 + \phi - 1) \mathcal{F}^{(1)}_2 (b_2 - 1) = 0;\\
  & b_1 \,	(c_1 + \theta - 1) \,  \mathcal{F}^{(1)}_2 (b_1 + 1) -  (c_1 - 1)  \, (b_1 + \theta) \mathcal{F}^{(1)}_2 (c_1 - 1) = 0;\\
  & b_1 \,	(c_2 + \phi - 1) \,  \mathcal{F}^{(1)}_2 (b_1 + 1) -  (c_2 - 1)  \, (b_1 + \theta) \mathcal{F}^{(1)}_2 (c_2 - 1) = 0;\\
  & b_1 \,	c_1 \,  \mathcal{F}^{(1)}_2 (b_1 + 1) -  (c_1 + \theta)  \, (b_1 + \theta) \mathcal{F}^{(1)}_2 (c_1 + 1) = 0;\\
  & b_1 \,	c_2 \,  \mathcal{F}^{(1)}_2 (b_1 + 1) -  (c_2 + \phi)  \, (b_1 + \theta) \mathcal{F}^{(1)}_2 (c_2 + 1) = 0;\\
   & b_2 \,	(b_1 - 1) \,  \mathcal{F}^{(1)}_2 (b_2 + 1) -  (b_2 + \phi)  \, (b_1 + \theta - 1) \mathcal{F}^{(1)}_2 (b_1 - 1) = 0;\\
  & b_2 \,	(b_2 - 1) \,  \mathcal{F}^{(1)}_2 (b_2 + 1) -  (b_2 + \phi)  \, (b_2 + \phi - 1) \mathcal{F}^{(1)}_2 (b_2 - 1) = 0;\\
  & b_2 \,	(c_1 + \theta - 1) \,  \mathcal{F}^{(1)}_2 (b_2 + 1) -  (c_1 - 1)  \, (b_2 + \phi) \mathcal{F}^{(1)}_2 (c_1 - 1) = 0;\\
  & b_2 \,	(c_2 + \phi - 1) \,  \mathcal{F}^{(1)}_2 (b_2 + 1) -  (c_2 - 1)  \, (b_2 + \phi) \mathcal{F}^{(1)}_2 (c_2 - 1) = 0;\\
  & b_2 \,	c_1 \,  \mathcal{F}^{(1)}_2 (b_2 + 1) -  (c_1 + \theta)  \, (b_2 + \phi) \mathcal{F}^{(1)}_2 (c_1 + 1) = 0;\\
  & b_2 \,	c_2 \,  \mathcal{F}^{(1)}_2 (b_2 + 1) -  (c_2 + \phi)  \, (b_2 + \phi) \mathcal{F}^{(1)}_2 (c_2 + 1) = 0;\\
  & (b_2 - 1) \,	(b_1 + \theta - 1) \,  \mathcal{F}^{(1)}_2 (b_1 - 1) -  (b_1 - 1)  \, (b_2 + \phi - 1) \mathcal{F}^{(1)}_2 (b_2 - 1) = 0;\\
  & (b_1 + \theta - 1) \,	(c_1 + \theta - 1) \,  \mathcal{F}^{(1)}_2 (b_1 - 1) -  (c_1 - 1)  \, (b_1 - 1) \mathcal{F}^{(1)}_2 (c_1 - 1) = 0;\\
  & (b_1 + \theta - 1) \,	(c_2 + \phi - 1) \,  \mathcal{F}^{(1)}_2 (b_1 - 1) -  (c_2 - 1)  \, (b_1 - 1) \mathcal{F}^{(1)}_2 (c_2 - 1) = 0;\\
  & 	c_1 \, (b_1 + \theta - 1) \, \mathcal{F}^{(1)}_2 (b_1 - 1) - (b_1 - 1)  (c_1 + \theta)  \,  \mathcal{F}^{(1)}_2 (c_1 + 1) = 0;\\
  & 	c_2 \, (b_1 + \theta - 1) \, \mathcal{F}^{(1)}_2 (b_1 - 1) - (b_1 - 1)  (c_2 + \phi)  \,  \mathcal{F}^{(1)}_2 (c_2 + 1) = 0;\\
  & (b_2 + \phi - 1) \,	(c_1 + \theta - 1) \,  \mathcal{F}^{(1)}_2 (b_2 - 1) -  (c_1 - 1)  \, (b_2 - 1) \mathcal{F}^{(1)}_2 (c_1 - 1) = 0;\\
  & (b_2 + \phi - 1) \,	(c_2 + \phi - 1) \,  \mathcal{F}^{(1)}_2 (b_2 - 1) -  (c_2 - 1)  \, (b_2 - 1) \mathcal{F}^{(1)}_2 (c_2 - 1) = 0;\\
  & 	c_1 \, (b_2 + \phi - 1) \, \mathcal{F}^{(1)}_2 (b_2 - 1) - (b_2 - 1)  (c_1 + \theta)  \,  \mathcal{F}^{(1)}_2 (c_1 + 1) = 0;\\
  & 	c_2 \, (b_2 + \phi - 1) \, \mathcal{F}^{(1)}_2 (b_2 - 1) - (b_2 - 1)  (c_2 + \phi)  \,  \mathcal{F}^{(1)}_2 (c_2 + 1) = 0;\\
   & 	c_1 \, (c_1 - 1) \, \mathcal{F}^{(1)}_2 (c_1 - 1) - (c_1 + \theta - 1)  (c_1 + \theta)  \,  \mathcal{F}^{(1)}_2 (c_1 + 1) = 0;\\
   & 	 (c_1 - 1) \, (c_2 + \phi - 1) \, \mathcal{F}^{(1)}_2 (c_1 - 1) - (c_2 - 1)  \, (c_1 + \theta - 1)    \mathcal{F}^{(1)}_2 (c_2 - 1) = 0;\\
   & 	c_2 \, (c_1 - 1) \, \mathcal{F}^{(1)}_2 (c_1 - 1) - (c_1 + \theta - 1)  (c_2 + \phi)  \,  \mathcal{F}^{(1)}_2 (c_2 + 1) = 0;\\
   & 	 (c_1 + \theta) \, (c_2 + \phi - 1) \, \mathcal{F}^{(1)}_2 (c_1 + 1) - c_1 \, (c_2 - 1)  \,  \mathcal{F}^{(1)}_2 (c_2 - 1) = 0;\\
   & 	c_2 \, (c_1 + \theta) \, \mathcal{F}^{(1)}_2 (c_1 + 1) - c_1 \, (c_2 + \phi)  \,  \mathcal{F}^{(1)}_2 (c_2 + 1) = 0;\\
   & 	c_2 \, (c_2 - 1) \, \mathcal{F}^{(1)}_2 (c_2 - 1) - (c_2 + \phi - 1)  (c_2 + \phi)  \,  \mathcal{F}^{(1)}_2 (c_2 + 1) = 0.
 \end{align}
Similarly using  the difference relations, 
\begin{align}
	&	a \, \mathcal{F}^{(1)}_2 (a + 1)  = \left(a + \frac{1}{k_1}\Theta_{t_1} + \frac{1}{k_2} \Theta_{t_2}\right) \, \mathcal{F}^{(1)}_2;\\
	& \left(a + \frac{1}{k_1}\Theta_{t_1} + \frac{1}{k_2} \Theta_{t_2} - 1\right) \, \mathcal{F}^{(1)}_2 (a - 1)  =	(a - 1) \, \mathcal{F}^{(1)}_2;\\
	&	b_1 \, \mathcal{F}^{(1)}_2 (b_1 + 1)  = \left(b_1 + \frac{1}{k_1} \Theta_{t_1}\right) \, \mathcal{F}^{(1)}_2;\\
	&\left(b_1 + \frac{1}{k_1} \Theta_{t_1} - 1\right) \, \mathcal{F}^{(1)}_2 (b_1 - 1)  =	(b_1 - 1) \, \mathcal{F}^{(1)}_2;\\
	&	b_2 \, \mathcal{F}^{(1)}_2 (b_2 + 1)  = \left(b_2 + \frac{1}{k_2} \, \Theta_{t_2}\right) \, \mathcal{F}^{(1)}_2;\\
	& \left(b_2 + \frac{1}{k_2} \, \Theta_{t_2} - 1\right) \, \mathcal{F}^{(1)}_2 (b_2 - 1)  =	(b_2 - 1) \, \mathcal{F}^{(1)}_2;\\
	&	(c_1 - 1) \, \mathcal{F}^{(1)}_2 (c_1 - 1)  = \left(c_1 + \frac{1}{k_1} \Theta_{t_1} - 1\right) \, \mathcal{F}^{(1)}_2;\\
	& \left(c_1 + \frac{1}{k_1} \Theta_{t_1}\right) \, \mathcal{F}^{(1)}_2 (c_1 + 1)  =	c_1 \, \mathcal{F}^{(1)}_2;\\
	&	(c_2 - 1) \, \mathcal{F}^{(1)}_2 (c_2 - 1)  = \left(c_2 + \frac{1}{k_2} \Theta_{t_2} - 1\right) \, \mathcal{F}^{(1)}_2;\\
	& \left(c_2 + \frac{1}{k_2} \Theta_{t_2}\right) \, \mathcal{F}^{(1)}_2 (c_2 + 1)  =	c_2 \, \mathcal{F}^{(1)}_2,
\end{align}
we will get the following difference recursion formulas obeyed by discrete Appell function $\mathcal{F}^{(1)}_2$. 
\begin{align}
	& a \, (a - 1) \, \mathcal{F}^{(1)}_2 (a + 1)\nonumber\\
	& \quad - \left(a + \frac{1}{k_1}\Theta_{t_1} + \frac{1}{k_2} \Theta_{t_2}\right) \, \left(a + \frac{1}{k_1}\Theta_{t_1} + \frac{1}{k_2} \Theta_{t_2} - 1\right) \mathcal{F}^{(1)}_2 (a - 1) = 0;\\
	& a \, (b_1 - 1) \, \mathcal{F}^{(1)}_2 (a + 1) \nonumber\\
	& \quad - \left(a + \frac{1}{k_1}\Theta_{t_1} + \frac{1}{k_2} \Theta_{t_2} \right) \, \left(b_1 + \frac{1}{k_1} \Theta_{t_1}  - 1\right) \mathcal{F}^{(1)}_2 (b_1 - 1) = 0;\\
	& a \, (b_2 - 1) \, \mathcal{F}^{(1)}_2 (a + 1)\nonumber\\
	& \quad - \left(a + \frac{1}{k_1}\Theta_{t_1} + \frac{1}{k_2} \Theta_{t_2}\right) \, \left(b_2 + \frac{1}{k_2} \Theta_{t_2} - 1\right) \mathcal{F}^{(1)}_1 (b_2 - 1) = 0;\\
	& a \, c_1 \, \mathcal{F}^{(1)}_2 (a + 1)\nonumber\\
&	\quad  - \left(a + \frac{1}{k_1}\Theta_{t_1} + \frac{1}{k_2} \Theta_{t_2}\right) \, \left(c_1 + \frac{1}{k_1} \Theta_{t_1} \right) \mathcal{F}^{(1)}_2 (c_1 + 1) = 0;\\
& a \, c_2 \, \mathcal{F}^{(1)}_2 (a + 1)\nonumber\\
&	\quad  - \left(a + \frac{1}{k_1}\Theta_{t_1} + \frac{1}{k_2} \Theta_{t_2}\right) \, \left(c_2 + \frac{1}{k_2} \Theta_{t_2} \right) \mathcal{F}^{(1)}_2 (c_2 + 1) = 0;\\
	& a \, \left(b_1 + \frac{1}{k_1} \Theta_{t_1}\right) \, \mathcal{F}^{(1)}_2 (a + 1) - b_1 \, \left(a + \frac{1}{k_1}\Theta_{t_1} + \frac{1}{k_2} \Theta_{t_2}\right) \, \mathcal{F}^{(1)}_2 (b_1 + 1) = 0;\\
	& a \, \left(b_2 + \frac{1}{k_2} \Theta_{t_2}\right) \, \mathcal{F}^{(1)}_2 (a + 1) - b_2 \, \left(a + \frac{1}{k_1}\Theta_{t_1} + \frac{1}{k_2} \Theta_{t_2}\right) \, \mathcal{F}^{(1)}_2 (b_2 + 1) = 0;\\
	& a \, \left(c_1 + \frac{1}{k_1} \Theta_{t_1} - 1\right) \, \mathcal{F}^{(1)}_2 (a + 1)\nonumber\\
	& \quad  - (c_1 - 1) \, \left(a + \frac{1}{k_1}\Theta_{t_1} + \frac{1}{k_2} \Theta_{t_2}\right) \, \mathcal{F}^{(1)}_2 (c_1 - 1) = 0;\\
	& a \, \left(c_2 + \frac{1}{k_2} \Theta_{t_2} - 1\right) \, \mathcal{F}^{(1)}_2 (a + 1)\nonumber\\
	& \quad  - (c_2 - 1) \, \left(a + \frac{1}{k_1}\Theta_{t_1} + \frac{1}{k_2} \Theta_{t_2}\right) \, \mathcal{F}^{(1)}_2 (c_2 - 1) = 0;\\
	&\left(a + \frac{1}{k_1}\Theta_{t_1} + \frac{1}{k_2} \Theta_{t_2} - 1\right) \, \left(b_1 + \frac{1}{k_1} \Theta_{t_1} \right) \, \mathcal{F}^{(1)}_2 (a - 1) - b_1 \, (a - 1)  \, \mathcal{F}^{(1)}_2 (b_1 + 1) = 0;\\
	&	\left(a + \frac{1}{k_1}\Theta_{t_1} + \frac{1}{k_2} \Theta_{t_2} - 1\right) \, \left(b_2 + \frac{1}{k_2}  \Theta_{t_2}\right) \, \mathcal{F}^{(1)}_2 (a - 1) - b_2 \, (a - 1)  \, \mathcal{F}^{(1)}_2 (b_2 + 1) = 0;\\
	&	\left(a + \frac{1}{k_1}\Theta_{t_1} + \frac{1}{k_2} \Theta_{t_2} - 1\right) \, \left(c_1 + \frac{1}{k_1} \Theta_{t_1} - 1\right) \, \mathcal{F}^{(1)}_2 (a - 1)\nonumber\\
	& \quad  - (c_1 - 1) \, (a - 1)  \, \mathcal{F}^{(1)}_2 (c_1 - 1) = 0;\\
	&	\left(a + \frac{1}{k_1}\Theta_{t_1} + \frac{1}{k_2} \Theta_{t_2} - 1\right) \, \left(c_2 + \frac{1}{k_2} \Theta_{t_2} - 1\right) \, \mathcal{F}^{(1)}_2 (a - 1)\nonumber\\
	& \quad  - (c_2 - 1) \, (a - 1)  \, \mathcal{F}^{(1)}_2 (c_2 - 1) = 0;\\
	& (b_1 - 1) \,	\left(a + \frac{1}{k_1}\Theta_{t_1} + \frac{1}{k_2} \Theta_{t_2} - 1\right) \,  \mathcal{F}^{(1)}_2 (a - 1) \nonumber\\
	& \quad -  (a - 1)  \, \left(b_1 + \frac{1}{k_1} \Theta_{t_1}  - 1\right) \mathcal{F}^{(1)}_2 (b_1 - 1) = 0;\\
	& (b_2 - 1) \,	\left(a + \frac{1}{k_1}\Theta_{t_1} + \frac{1}{k_2} \Theta_{t_2} - 1\right) \,  \mathcal{F}^{(1)}_2 (a - 1)\nonumber\\
	& \quad  -  (a - 1)  \, \left(b_2 + \frac{1}{k_2} \Theta_{t_2} - 1\right) \mathcal{F}^{(1)}_2 (b_2 - 1) = 0;\\
	& c_1 \,	\left(a + \frac{1}{k_1}\Theta_{t_1} + \frac{1}{k_2} \Theta_{t_2} - 1\right) \,  \mathcal{F}^{(1)}_2 (a - 1)\nonumber\\
	& \quad -  (a - 1)  \, \left(c_1 + \frac{1}{k_1} \Theta_{t_1}\right) \mathcal{F}^{(1)}_2 (c_1 + 1) = 0;\\
	& c_2 \,	\left(a + \frac{1}{k_1}\Theta_{t_1} + \frac{1}{k_2} \Theta_{t_2} - 1\right) \,  \mathcal{F}^{(1)}_2 (a - 1)\nonumber\\
	& \quad -  (a - 1)  \, \left(c_2 + \frac{1}{k_2} \Theta_{t_2}\right) \mathcal{F}^{(1)}_2 (c_2 + 1) = 0;\\
	& b_1 \,	(b_1 - 1) \,  \mathcal{F}^{(1)}_2 (b_1 + 1) -  \left(b_1 + \frac{1}{k_1} \Theta_{t_1} \right)  \, \left(b_1 + \frac{1}{k_1} \Theta_{t_1}  - 1\right) \mathcal{F}^{(1)}_2 (b_1 - 1) = 0;\\
	& b_1 \,	\left(b_2 + \frac{1}{k_2} \Theta_{t_2}\right) \,  \mathcal{F}^{(1)}_2 (b_1 + 1) -  b_2  \, \left(b_1 + \frac{1}{k_1} \Theta_{t_1}\right) \mathcal{F}^{(1)}_2 (b_2 + 1) = 0;\\
	& b_1 \,	(b_2 - 1) \,  \mathcal{F}^{(1)}_2 (b_1 + 1) -  \left(b_1 + \frac{1}{k_1} \Theta_{t_1}\right)  \, \left(b_2 + \frac{1}{k_2} \Theta_{t_2} - 1\right) \mathcal{F}^{(1)}_2 (b_2 - 1) = 0;\\
	& b_1 \,	\left(c_1 + \frac{1}{k_1} \Theta_{t_1} -  1\right) \,  \mathcal{F}^{(1)}_2 (b_1 + 1) -  (c_1 - 1)  \, \left(b_1 + \frac{1}{k_1} \Theta_{t_1} \right) \mathcal{F}^{(1)}_2 (c_1 - 1) = 0;\\
	& b_1 \,	\left(c_2 + \frac{1}{k_2} \Theta_{t_2} -  1\right) \,  \mathcal{F}^{(1)}_2 (b_1 + 1) -  (c_2 - 1)  \, \left(b_1 + \frac{1}{k_1} \Theta_{t_1} \right) \mathcal{F}^{(1)}_2 (c_2 - 1) = 0;\\
	& b_1 \,	c_1 \,  \mathcal{F}^{(1)}_2 (b_1 + 1) -  \left(c_1 + \frac{1}{k_1} \Theta_{t_1}\right)  \, \left(b_1 + \frac{1}{k_1} \Theta_{t_1}\right) \mathcal{F}^{(1)}_2 (c_1 + 1) = 0;\\
	& b_1 \,	c_2 \,  \mathcal{F}^{(1)}_2 (b_1 + 1) -  \left(c_2 + \frac{1}{k_2} \Theta_{t_2}\right)  \, \left(b_1 + \frac{1}{k_1} \Theta_{t_1}\right) \mathcal{F}^{(1)}_2 (c_2 + 1) = 0;\\
	& b_2 \,	(b_1 - 1) \,  \mathcal{F}^{(1)}_2 (b_2 + 1) -  \left(b_2 + \frac{1}{k_2} \Theta_{t_2}\right)  \, \left(b_1 + \frac{1}{k_1} \Theta_{t_1}  - 1\right) \mathcal{F}^{(1)}_2 (b_1 - 1) = 0;\\
	& b_2 \,	(b_2 - 1) \,  \mathcal{F}^{(1)}_2 (b_2 + 1) -  \left(b_2 + \frac{1}{k_2} \Theta_{t_2}\right)  \, \left(b_2 + \frac{1}{k_2} \Theta_{t_2}  - 1\right) \mathcal{F}^{(1)}_2 (b_2 - 1) = 0;\\
	& b_2 \,	\left(c_1 + \frac{1}{k_1} \Theta_{t_1} - 1\right) \,  \mathcal{F}^{(1)}_2 (b_2 + 1) -  (c_1 - 1)  \, \left(b_2 + \frac{1}{k_2}  \Theta_{t_2}\right) \mathcal{F}^{(1)}_2 (c_1 - 1) = 0;\\
	& b_2 \,	\left(c_2 + \frac{1}{k_2} \Theta_{t_2} - 1\right) \,  \mathcal{F}^{(1)}_2 (b_2 + 1) -  (c_2 - 1)  \, \left(b_2 + \frac{1}{k_2}  \Theta_{t_2}\right) \mathcal{F}^{(1)}_2 (c_2 - 1) = 0;\\
	& b_2 \,	c_1 \,  \mathcal{F}^{(1)}_2 (b_2 + 1) -  \left(c_1 + \frac{1}{k_1} \Theta_{t_1}\right)  \, \left(b_2 + \frac{1}{k_2} \Theta_{t_2}\right) \mathcal{F}^{(1)}_2 (c_1 + 1) = 0;\\
	& b_2 \,	c_2 \,  \mathcal{F}^{(1)}_2 (b_2 + 1) -  \left(c_2 + \frac{1}{k_2} \Theta_{t_2}\right)  \, \left(b_2 + \frac{1}{k_2} \Theta_{t_2}\right) \mathcal{F}^{(1)}_2 (c_2 + 1) = 0;\\
	& (b_2 - 1) \,	\left(b_1 + \frac{1}{k_1} \Theta_{t_1}  - 1\right) \,  \mathcal{F}^{(1)}_2 (b_1 - 1) -  (b_1 - 1)  \, \left(b_2 + \frac{1}{k_2}  \Theta_{t_2} - 1\right) \mathcal{F}^{(1)}_2 (b_2 - 1) = 0;\\
	& \left(b_1 + \frac{1}{k_1} \Theta_{t_1} - 1\right) \,	\left(c_1 + \frac{1}{k_1} \Theta_{t_1} - 1\right) \,  \mathcal{F}^{(1)}_2 (b_1 - 1)\nonumber\\
	& \quad -  (c_1 - 1)  \, (b_1 - 1) \mathcal{F}^{(1)}_2 (c_1 - 1) = 0;\\
	& \left(b_1 + \frac{1}{k_1} \Theta_{t_1} - 1\right) \,	\left(c_2 + \frac{1}{k_2} \Theta_{t_2} - 1\right) \,  \mathcal{F}^{(1)}_2 (b_1 - 1)\nonumber\\
	& \quad -  (c_2 - 1)  \, (b_1 - 1) \mathcal{F}^{(1)}_2 (c_2 - 1) = 0;\\
	& 	c_1 \, \left(b_1 + \frac{1}{k_1} \Theta_{t_1} - 1\right) \, \mathcal{F}^{(1)}_2 (b_1 - 1) - (b_1 - 1)  \left(c_1 + \frac{1}{k_1} \Theta_{t_1}\right)  \,  \mathcal{F}^{(1)}_2 (c_1 + 1) = 0;\\
	& 	c_2 \, \left(b_1 + \frac{1}{k_1} \Theta_{t_1} - 1\right) \, \mathcal{F}^{(1)}_2 (b_1 - 1) - (b_1 - 1)  \left(c_2 + \frac{1}{k_2} \Theta_{t_2}\right)  \,  \mathcal{F}^{(1)}_2 (c_2 + 1) = 0;\\
	& \left(b_2 + \frac{1}{k_2} \Theta_{t_2} - 1\right) \,	\left(c_1 + \frac{1}{k_1} \Theta_{t_1} - 1\right) \,  \mathcal{F}^{(1)}_2 (b_2 - 1)\nonumber\\
	& \quad -  (c_1 - 1)  \, (b_2 - 1) \mathcal{F}^{(1)}_2 (c_1 - 1) = 0;\\
	& \left(b_2 + \frac{1}{k_2} \Theta_{t_2} - 1\right) \,	\left(c_2 + \frac{1}{k_2} \Theta_{t_2} - 1\right) \,  \mathcal{F}^{(1)}_2 (b_2 - 1)\nonumber\\
	& \quad -  (c_2 - 1)  \, (b_2 - 1) \mathcal{F}^{(1)}_2 (c_2 - 1) = 0;\\
	& 	c_1 \, \left(b_2 + \frac{1}{k_2}  \Theta_{t_2} - 1\right) \, \mathcal{F}^{(1)}_2 (b_2 - 1) - (b_2 - 1)  \left(c_1 + \frac{1}{k_1} \Theta_{t_1}\right)  \,  \mathcal{F}^{(1)}_2 (c_1 + 1) = 0;\\
	& 	c_2 \, \left(b_2 + \frac{1}{k_2}  \Theta_{t_2} - 1\right) \, \mathcal{F}^{(1)}_2 (b_2 - 1) - (b_2 - 1)  \left(c_2 + \frac{1}{k_2} \Theta_{t_2}\right)  \,  \mathcal{F}^{(1)}_2 (c_2 + 1) = 0;\\
	& 	c_1 \, (c_1 - 1) \, \mathcal{F}^{(1)}_2 (c_1 - 1) - \left(c_1 + \frac{1}{k_1} \Theta_{t_1} - 1\right)  \left(c_1 + \frac{1}{k_1} \Theta_{t_1}\right)  \,  \mathcal{F}^{(1)}_2 (c_1 + 1) = 0;\\
	& 	 (c_1 - 1) \, \left(c_2 + \frac{1}{k_2} \Theta_{t_2} - 1\right) \, \mathcal{F}^{(1)}_2 (c_1 - 1) - (c_2 - 1)  \, \left(c_1 + \frac{1}{k_1} \Theta_{t_1} - 1\right)    \mathcal{F}^{(1)}_2 (c_2 - 1) = 0;\\
	& 	c_2 \, (c_1 - 1) \, \mathcal{F}^{(1)}_2 (c_1 - 1) - \left(c_1 + \frac{1}{k_1} \Theta_{t_1} - 1\right)  \left(c_2 + \frac{1}{k_2} \Theta_{t_2}\right) \,  \mathcal{F}^{(1)}_2 (c_2 + 1) = 0;\\
	& 	 \left(c_1 + \frac{1}{k_1} \Theta_{t_1}\right) \, \left(c_2 + \frac{1}{k_2} \Theta_{t_2} - 1\right) \, \mathcal{F}^{(1)}_2 (c_1 + 1) - c_1 \, (c_2 - 1)  \,  \mathcal{F}^{(1)}_2 (c_2 - 1) = 0;\\
	& 	c_2 \, \left(c_1 + \frac{1}{k_1} \Theta_{t_1} \right) \, \mathcal{F}^{(1)}_2 (c_1 + 1) - c_1 \, \left(c_2 + \frac{1}{k_2} \Theta_{t_2}\right)  \,  \mathcal{F}^{(1)}_2 (c_2 + 1) = 0;\\
	& 	c_2 \, (c_2 - 1) \, \mathcal{F}^{(1)}_2 (c_2 - 1) - \left(c_2 + \frac{1}{k_2} \Theta_{t_2} - 1\right)  \left(c_2 + \frac{1}{k_2} \Theta_{t_2}\right)  \,  \mathcal{F}^{(1)}_2 (c_2 + 1) = 0.
\end{align}

\section{Discrete  Appell function $\mathcal{F}^{(2)}_2$}
As given in the introduction that there are two principal discrete forms of Appell function $F_2$. In the preceeding sections, we elaborated the discrete Appell function $\mathcal{F}^{(1)}_2$ rigorously. In this section, we acquire results related to discrete function $\mathcal{F}^{(2)}_2$. The way we approach in findings of $\mathcal{F}^{(2)}_2$ is similar to $\mathcal{F}^{(1)}_2$, so omit the proof and produced the results directly.  

The discrete Appell function $\mathcal{F}^{(2)}_2$ converges absolutely when $\vert x\vert + \vert y\vert < 1$. For different values of $k$, the discrete Appell function $\mathcal{F}^{(2)}_2$ reduces into special complex functions. In particular, for $k = 0$:
\begin{align}
	\mathcal{F}^{(2)}_2(a, b_1, b_2; c_1, c_2; t, 0, x, y)
	% & = \sum_{m, n \geq 0} \frac{(a)_{m+n} \, (b_1)_m \, (b_2)_n}{ (c_1)_{m} \, (c_2)_n \, m! \, n!} \ x^m \, y^n\nonumber\\
	 = F_2 (a, b_1, b_2; c_1, c_2; x, y).
\end{align}
For $k = 1$, the discrete function $\mathcal{F}^{(2)}_2$ turns out into a Kamp\'e de F\'eriet function as:
\begin{align}
	 \mathcal{F}^{(2)}_2(a, b_1, b_2; c_1, c_2; t, 1, x, y) 
%	& = \sum_{m, n \geq 0} \frac{(a)_{m+n} \, (b_1)_m \, (b_2)_n \, (-1)^{(m + n)} \, (-t)_{m + n}}{ (c_1)_{m} \, (c_2)_n \, m! \, n!} \ x^m \, y^n\nonumber\\
	= F_{{0}:{1}, {1}} ^{{2}:{1}, {1}}\left(\begin{array}{ccc}
		a, -t: & b_1 & b_2\\
		-: & c_1 & c_2 
	\end{array}; -x, \, -y\right).
\end{align}
The limiting cases of discrete Appell function $\mathcal{F}^{(2)}_2$ give the second discrete analogues of Humbert function $\psi_1$ and $\psi_2$. We denote these two discrete Humbert functions by $\psi^{(2)}_1$, $\psi^{(2)}_2$; defined as
\begin{align}
	\psi^{(2)}_1 \left(a, b_1; c_1, c_2; t, k, x,  y\right) & = \sum_{m,n \geq 0} \frac{(a)_{m+n} \, (b_1)_m  \, (-1)^{(m + n) k} \, (-t)_{(m + n)k}}{ (c_1)_{m} \, (c_2)_n \, m! \, n!} \ x^m \, y^n ,\nonumber
	\\[5pt]
	\psi^{(2)}_2 \left( a; c_1, c_2; t, k, x,  y\right) & = \sum_{m,n \geq 0} \frac{(a)_{m+n} \, (-1)^{(m + n) k} \, (-t)_{(m + n)k}}{ (c_1)_{m} \, (c_2)_n \, m! \, n!} \ x^m \, y^n.
\end{align}
This can be verified as 
\begin{align}
	 \lim_{\varepsilon \to 0}  \mathcal{F}^{(2)}_2 \left(a, b_1, \frac{1}{\varepsilon}; c_1, c_2; t, k, x, \varepsilon \, y\right) 	& = \psi^{(1)}_1 \left(a, b_1; c_1, c_2; t, k, x,  y\right)
\end{align}
and
\begin{align}
	\lim_{\varepsilon \to 0}  \mathcal{F}^{(2)}_2 \left(a, \frac{1}{\varepsilon}, \frac{1}{\varepsilon} ; c_1, c_2; t, k, \varepsilon \, x, \varepsilon \, y\right) = \psi^{(1)}_2 \left( a; c_1, c_2; t, k, x,  y\right).
\end{align}
The discrete Appell function $F_2^{(2)}$ satisfies the following difference-differential equations
\begin{align}
	\left(\theta \left(\theta + c_1 - 1\right) - k \, (-1)^k \, (-t)_k \, x \, \rho_t^k \,  \left(\frac{1}{k} \Theta_t + a\right) \, \left(\theta  + b_1\right) \right) \mathcal{F}^{(2)}_2 = 0.
\end{align}	    
\begin{align}
	\left(\phi \, \left(\phi + c_2 - 1\right) - k \, (-1)^k \, (-t)_k \, y \, \rho_t^k \,  \left(\frac{1}{k} \Theta_t + a\right) \, \left(\phi  + b_2 \right) \right) \mathcal{F}^{(2)}_2 = 0.
\end{align}	    
%and
%\begin{align}
%	\left[(-t_2)_k \, y \, \Theta_{t_1} \, \left(\frac{1}{k} \Theta_{t_1}  + c_1 - 1\right) \left(\frac{1}{k} \Theta_{t_2}  + b_2\right) - (-t_1)_k \, x \, \Theta_{t_2} \, \left(\frac{1}{k} \Theta_{t_2}  + c_2 - 1\right) \left(\frac{1}{k} \Theta_{t_1}  + b_1\right)\right] \,  \mathcal{F}^{(2)}_2 = 0. 
%\end{align}
\subsection{Integral representations}
We now find the integral representations of the discrete Appell  function $\mathcal{F}^{(2)}_2$.
\begin{theorem}
	Let $a$, $b_1$, $b_2$, $c_1$, $c_2$ and $t$ be complex numbers. Then for $\vert x\vert  + \vert y\vert < 1$, the discrete function defined in \eqref{3.2} can be represented in the integral forms as
	\begin{align}
		& \mathcal{F}^{(2)}_2(a, b_1, b_2; c_1, c_2; t, k, x, y)\nonumber\\
		& =\Gamma \left(\begin{array}{c}
			c_1, c_2\\
			b_1, b_2, c_1 - b_1, c_2 - b_2
		\end{array}\right) \int_{0}^{1} \int_{0}^{1} u^{b_1 - 1} v^{b_2 - 1} (1-u)^{c_1 -  b_1 - 1} (1 - v)^{c_2 - b_2 - 1} \nonumber\\
		& \quad \times F_{{0}:{0}; {0}} ^{{k + 1}:{0}; {0}}\left(\begin{array}{ccc}
			a, \frac{- t}{k}, \dots, \frac{- t + k - 1}{k}: & -  ; & -\\
			-: & - ; & - 
		\end{array}; (-k)^k \, u x,  (-k)^k \, v y\right) du \, dv;\nonumber\\
	& = \frac{1}{\Gamma (a)} \int_{0}^{\infty} e^{-u} \, u^{a - 1} \nonumber\\
	& \quad \times F_{{0}:{1}; {1}} ^{{k}:{1}; {1}}\left(\begin{array}{ccc}
		\frac{- t}{k}, \dots, \frac{- t + k - 1}{k} : & b_1; & b_2\\
		: & c_1; & c_2
	\end{array}; (-k)^k \, u x,  (-k)^k \, u y\right) du;\nonumber\\
	& = \frac{1}{\Gamma (b_1)} \int_{0}^{\infty} e^{-u} \, u^{b_1 - 1} \nonumber\\
	& \quad \times F_{{0}:{1}; {1}} ^{{k + 1}:{0}; { 1}}\left(\begin{array}{ccc}
		a, \frac{- t}{k}, \dots, \frac{- t + k - 1}{k}: & -; & b_2\\
		-: & c_1 ; & c_2 
	\end{array}; (-k)^k \, u x,  (-k)^k \, y\right) du;\\
	& = \frac{1}{\Gamma (b_2)} \int_{0}^{\infty} e^{-v} \, v^{b_2 - 1} \nonumber\\
	& \quad \times F_{{0}:{1}; {1}} ^{{k + 1}:{1}; {0}}\left(\begin{array}{ccc}
		a, \frac{- t}{k}, \dots, \frac{- t + k - 1}{k}: & b_1; & -\\
		-: & c_1; & c_2 
	\end{array}; (-k)^k \, x,  (-k)^k \, v y\right) dv;\\
& = \frac{1}{\Gamma (-t)} \int_{0}^{\infty} e^{-v} \, v^{-t - 1}  F_{{0}:{1}; {1}} ^{{1}:{1}; {1}}\left(\begin{array}{ccc}
a: & b_1; & b_2\\
-: & c_1; & c_2 
\end{array}; (-v)^k \, x,  (-v)^k \, y\right) dv.
\end{align}

\end{theorem}
\subsection{Differential formulae}
We have the following differential formulae satisfied by discrete Appell function $\mathcal{F}^{(2)}_2$:
\begin{align}
%	&	(\Delta_{t})^r \mathcal{F}^{(2)}_2(a, b_1, b_2; c_1, c_2; t, k, x, y) \nonumber\\
%	& = \frac{(a)_r \, (b_1)_r \, k^r \, x^r \, (-1)^{r(k - 1)} (-t)_{r (k - 1)}}{(c_1)_r} \nonumber\\
%	& \quad \times  \mathcal{F}^{(2)}_2(a + r, b_1 + r, b_2; c_1 + r, c_2; t - r(k - 1), k, x, y);\nonumber\\
%	& (\Theta_{t})^r \mathcal{F}^{(2)}_2(a, b_1, b_2; c_1, c_2; t, k, x, y) \nonumber\\
%	& = \frac{(-1)^{rk} \, (a)_r \, (b_1)_r \, (-t)_{rk} \, x^r \, k^r}{(c_1)_r} \, \mathcal{F}^{(2)}_2(a + r, b_1 + r, b_2; c_1 + r, c_2; t - rk, k, x, y);\nonumber\\
	& (\theta)^r \mathcal{F}^{(2)}_2(a, b_1, b_2; c_1, c_2; t, k, x, y) \nonumber\\
	& = \frac{(-1)^{rk} \, (a)_r \, (b_1)_r \, (-t)_{rk} \, x^r}{(c_1)_r} \, \mathcal{F}^{(2)}_2 (a + r, b_1 + r, b_2; c_1 + r, c_2; t - rk, k, x, y);\nonumber\\
	& (\phi)^r \mathcal{F}^{(2)}_2(a, b_1, b_2; c_1, c_2; t, k, x, y) \nonumber\\
	& = \frac{(-1)^{rk} \, (a)_r \, (b_2)_r \, (-t)_{rk} \, y^r}{(c_2)_r} \, \mathcal{F}^{(2)}_2(a + r, b_1, b_2 + r; c_1, c_2 + r; t - rk, k, x, y);\\
	& \left(\frac{\partial}{\partial x}\right)^r \left[x^{b_1 + r - 1} \mathcal{F}^{(2)}_2(a, b_1, b_2; c_1, c_2; t, k, x, y)\right]\nonumber\\
	& = x^{b_1 - 1} \, (b_1)_r \mathcal{F}^{(2)}_2(a, b_1 + r, b_2; c_1, c_2; t, k, x, y);\nonumber\\
	& \left(\frac{\partial}{\partial y}\right)^r [y^{b_2 + r - 1} \mathcal{F}^{(2)}_2(a, b_1, b_2; c_1, c_2; t, k, x, y)]\nonumber\\
	& = y^{b_2 - 1} \, (b_2)_r \mathcal{F}^{(2)}_2 (a, b_1, b_2 + r; c_1, c_2; t, k, x, y);\\
	& \left(\frac{\partial}{\partial x}\right)^r [x^{a + r - 1} \mathcal{F}^{(2)}_2(a, b_1, b_2; c_1, c_2; t, k, x, xy)]\nonumber\\
	& = x^{a - 1} \, (a)_r \mathcal{F}^{(2)}_2 (a + r, b_1, b_2; c_1, c_2; t, k, x, xy);\\
	& \left(\frac{\partial}{\partial y}\right)^r [y^{a + r - 1} \mathcal{F}^{(2)}_2 (a, b_1, b_2; c_1, c_2; t, k, xy, y)]\nonumber\\
	& = y^{a - 1} \, (a)_r \mathcal{F}^{(2)}_2(a + r, b_1, b_2; c_1, c_2; t, k, xy, y);\\
	& \left(\frac{\partial}{\partial x}\right)^r [x^{c_1 - 1} \mathcal{F}^{(2)}_2(a, b_1, b_2; c_1, c_2; t, k, x, y)]\nonumber\\
	& = (-1)^r \, (1 - c_1)_r \, x^{c_1 - r - 1} \mathcal{F}^{(2)}_2(a, b_1, b_2; c_1 - r, c_2; t, k, x, y);\\
	& \left(\frac{\partial}{\partial y}\right)^r [y^{c_2 - 1} \mathcal{F}^{(2)}_2(a, b_1, b_2; c_1, c_2; t, k, x, y)]\nonumber\\
	& = (-1)^r \, y^{c_2 - r - 1} \, (1 - c_2)_r \mathcal{F}^{(2)}_2(a, b_1, b_2; c_1, c_2 - r; t, k, x, y).
\end{align}
\subsection{Finite and infinite summation formulae}
Now, we establish some finite and infinite summation formulae in terms of discrete Appell function $\mathcal{F}^{(2)}_2$ in the following theorem.
\begin{theorem}
	\begin{align}
		& \mathcal{F}^{(2)}_2(a, b_1 + r, b_2; c_1, c_2; t, k, x, y)\nonumber\\
		& = \sum_{s = 0}^{r} {r \choose s} \frac{(a)_s \, (-1)^{sk} \, (-t_1)_{sk}}{(c_1)_s} \, x^s \, \mathcal{F}^{(2)}_2(a + s, b_1 + s, b_2; c_1 + s, c_2; t - sk, t_2, k, x, y);\nonumber\\
		& \mathcal{F}^{(2)}_2(a, b_1, b_2 + r; c_1, c_2; t, k, x, y)\nonumber\\
		& = \sum_{s = 0}^{r} {r \choose s} \frac{(a)_s \, (-1)^{sk} \, (-t_2)_{sk}}{(c_2)_s} \, y^s \, \mathcal{F}^{(2)}_2(a + s, b_1, b_2 + s; c_1, c_2 + s; t - sk, k, x, y).\label{5.2}
	\end{align}
	\begin{align}
		&\sum_{r = 0}^{\infty} \frac{(a)_r}{r !} \, z^r \, \mathcal{F}^{(2)}_2 (a + r, b_1, b_2; c_1, c_2 ; t, k, x, y)\nonumber\\
		& = (1 - z)^{-a} \, \mathcal{F}^{(2)}_2 \left(a, b_1, b_2; c_1, c_2; t, k, \frac{x}{1 - z}, \frac{y}{1 - z}\right)\\
		&\sum_{r = 0}^{\infty} \frac{(b_1)_r}{r !} \, z^r \, \mathcal{F}^{(2)}_2 (a, b_1 + r, b_2; c_1, c_2 ; t, k, x, y)\nonumber\\
		& = (1 - z)^{-b_1} \, \mathcal{F}^{(2)}_2 \left(a, b_1, b_2; c_1, c_2; t, k, \frac{x}{1 - z}, y\right);\\
		&\sum_{r = 0}^{\infty} \frac{(b_2)_r}{r !} \, z^r \, \mathcal{F}^{(2)}_2 (a, b_1, b_2 + r; c_1, c_2 ; t, k, x, y)\nonumber\\
		& = (1 - z)^{-b_2} \, \mathcal{F}^{(1)}_2 \left(a, b_1, b_2; c_1, c_2; t, k, x, \frac{y}{1 - z}\right). 
	\end{align}
\end{theorem}
\subsection{Recursion Formulae}
We have the following recursion formulae satisfied by the discrete Appell function  $\mathcal{F}^{(2)}_2$:
\begin{align}
	& \mathcal{F}^{(2)}_2 (a + s, b_1, b_2; c_1, c_2 ; t, k, x, y) \nonumber\\
	& = \mathcal{F}^{(2)}_2 (a, b_1, b_2; c_1, c_2; t, k, x, y) \nonumber\\
	& \quad + \frac{(-1)^k \, (-t)_k \, b_1 \, x}{c_1} \, \sum_{r = 1}^{s} \mathcal{F}^{(2)}_2 (a + r, b_1 + 1, b_2; c_1 + 1, c_2; t - k, k, x, y)\nonumber\\
	& \quad + \frac{(-1)^k \, (-t)_k \, b_2 \, y}{c_2} \, \sum_{r = 1}^{s} \mathcal{F}^{(2)}_2 (a + r, b_1, b_2 + 1; c_1, c_2 + 1; t - k, k, x, y);\nonumber\\
	& \mathcal{F}^{(2)}_2 (a - s, b_1, b_2; c_1, c_2 ; t, k, x, y) \nonumber\\
	& = \mathcal{F}^{(2)}_2 (a, b_1, b_2; c_1, c_2 ; t, k, x, y)\nonumber\\
	& \quad - \frac{(-1)^k \, (-t)_k \, b_1 \, x}{c_1} \, \sum_{r = 0}^{s - 1} \mathcal{F}^{(2)}_2 (a - r, b_1 + 1, b_2; c_1 + 1, c_2; t - k, k, x, y)\nonumber\\
	& \quad - \frac{(-1)^k \, (-t)_k \, b_2 \, y}{c_2} \, \sum_{r = 0}^{s - 1} \mathcal{F}^{(2)}_2 (a - r, b_1, b_2 + 1; c_1, c_2 + 1; t - k, k, x, y);\nonumber\\
	& \mathcal{F}^{(2)}_2 (a, b_1 + s, b_2; c_1, c_2 ; t, k, x, y) \nonumber\\
	& = \mathcal{F}^{(2)}_2 (a, b_1, b_2; c_1, c_2 ; t, k, x, y)\nonumber\\
	& \quad + \frac{(-1)^k \, (-t)_k \, a \, x}{c_1} \, \sum_{r = 1}^{s} \mathcal{F}^{(2)}_2 (a + 1, b_1 + r, b_2; c_1 + 1, c_2; t - k, k, x, y);\nonumber\\
	& \mathcal{F}^{(2)}_2 (a, b_1 - s, b_2; c_1, c_2 ; t, k, x, y) \nonumber\\
	& = \mathcal{F}^{(2)}_2 (a, b_1, b_2; c_1, c_2; t, k, x, y) \nonumber\\
	& \quad - \frac{(-1)^k \, (-t)_k \, a \, x}{c_1} \, \sum_{r = 0}^{s - 1} \mathcal{F}^{(2)}_2 (a + 1, b_1 - r, b_2; c_1 + 1, c_2; t - k, k, x, y);\nonumber\\
	& \mathcal{F}^{(2)}_2 (a, b_1, b_2; c_1 - s, c_2; t, k, x, y) \nonumber\\
	& = \mathcal{F}^{(2)}_2 (a, b_1, b_2; c_1, c_2 ; t, k, x, y) \nonumber\\
	& \quad + (-1)^k \, (-t)_k \, a \, b_1 \, x \, \sum_{r = 1}^{s} \frac{\mathcal{F}^{(2)}_2 (a + 1, b_1 + 1, b_2; c_1 + 2 - r, c_2; t - k, k, x, y)}{(c_1 - r) \, (c_1 - r + 1)}\nonumber\\
	& \quad + (-1)^k \, (-t)_k \, a \, b_2 \, y \, \sum_{r = 1}^{s} \frac{\mathcal{F}^{(2)}_2 (a + 1, b_1, b_2 + 1; c_1 + 2 - r, c_2; t - k, k, x, y)}{(c_1 - r) \, (c_1 - r + 1)}.
\end{align}
In the parallel of $\mathcal{F}^{(1)}_2$, we give below a list of first and second order recursion relations satisfied by $\mathcal{F}^{(2)}_2$:
\begin{align}
	& a \, (a - 1) \, \mathcal{F}^{(2)}_2 (a + 1) - (a + \theta + \phi) \, (a + \theta + \phi - 1) \mathcal{F}^{(2)}_2 (a - 1) = 0;\\
	& a \, (b_1 - 1) \, \mathcal{F}^{(2)}_2 (a + 1) - (a + \theta + \phi) \, (b_1 + \theta  - 1) \mathcal{F}^{(2)}_2 (b_1 - 1) = 0;\\
	& a \, (b_2 - 1) \, \mathcal{F}^{(2)}_2 (a + 1) - (a + \theta + \phi) \, (b_2 + \phi - 1) \mathcal{F}^{(2)}_2 (b_2 - 1) = 0;\\
	& a \, c_1 \, \mathcal{F}^{(2)}_2 (a + 1) - (a + \theta + \phi) \, (c_1 + \theta) \mathcal{F}^{(2)}_2 (c_1 + 1) = 0;\\
	& a \, c_2 \, \mathcal{F}^{(2)}_2 (a + 1) - (a + \theta + \phi) \, (c_2 + \phi) \mathcal{F}^{(2)}_2 (c_2 + 1) = 0;\\
	& a \, (b_1 + \theta) \, \mathcal{F}^{(2)}_2 (a + 1) - b_1 \, (a + \theta + \phi) \, \mathcal{F}^{(2)}_2 (b_1 + 1) = 0;\\
	& a \, (b_2 + \phi) \, \mathcal{F}^{(2)}_2 (a + 1) - b_2 \, (a + \theta + \phi) \, \mathcal{F}^{(2)}_2 (b_2 + 1) = 0;\\
	& a \, (c_1 + \theta - 1) \, \mathcal{F}^{(2)}_2 (a + 1) - (c_1 - 1) \, (a + \theta + \phi) \, \mathcal{F}^{(2)}_2 (c_1 - 1) = 0;\\
	& a \, (c_2 + \phi - 1) \, \mathcal{F}^{(2)}_2 (a + 1) - (c_2 - 1) \, (a + \theta + \phi) \, \mathcal{F}^{(2)}_2 (c_2 - 1) = 0;\\
	&(a + \theta + \phi - 1) \, (b_1 + \theta) \, \mathcal{F}^{(2)}_2 (a - 1) - b_1 \, (a - 1)  \, \mathcal{F}^{(2)}_2 (b_1 + 1) = 0;\\
	&	(a + \theta + \phi - 1) \, (b_2 + \phi) \, \mathcal{F}^{(2)}_2 (a - 1) - b_2 \, (a - 1)  \, \mathcal{F}^{(2)}_2 (b_2 + 1) = 0;\\
	&	(a + \theta + \phi - 1) \, (c_1 + \theta - 1) \, \mathcal{F}^{(2)}_2 (a - 1) - (c_1 - 1) \, (a - 1)  \, \mathcal{F}^{(2)}_2 (c_1 - 1) = 0;\\
	& (a + \theta + \phi - 1) \, (c_2 + \phi - 1) \, \mathcal{F}^{(2)}_2 (a - 1) - (c_2 - 1) \, (a - 1)  \, \mathcal{F}^{(2)}_2 (c_2 - 1) = 0;\\
	& (b_1 - 1) \,	(a + \theta + \phi - 1) \,  \mathcal{F}^{(2)}_2 (a - 1) -  (a - 1)  \, (b_1 + \theta - 1) \mathcal{F}^{(2)}_2 (b_1 - 1) = 0;\\
	& (b_2 - 1) \,	(a + \theta + \phi - 1) \,  \mathcal{F}^{(2)}_2 (a - 1) -  (a - 1)  \, (b_2 + \phi - 1) \mathcal{F}^{(2)}_2 (b_2 - 1) = 0;\\
	& c_1 \,	(a + \theta + \phi - 1) \,  \mathcal{F}^{(2)}_2 (a - 1) -  (a - 1)  \, (c_1 + \theta) \mathcal{F}^{(2)}_2 (c_1 + 1) = 0;\\
	& c_2 \,	(a + \theta + \phi - 1) \,  \mathcal{F}^{(2)}_2 (a - 1) -  (a - 1)  \, (c_2 + \phi) \mathcal{F}^{(2)}_2 (c_2 + 1) = 0;\\
	& b_1 \,	(b_1 - 1) \,  \mathcal{F}^{(2)}_2 (b_1 + 1) -  (b_1 + \theta)  \, (b_1 + \theta - 1) \mathcal{F}^{(2)}_2 (b_1 - 1) = 0;\\
	& b_1 \,	(b_2 + \phi) \,  \mathcal{F}^{(2)}_2 (b_1 + 1) -  b_2  \, (b_1 + \theta ) \mathcal{F}^{(2)}_2 (b_2 + 1) = 0;\\
	& b_1 \,	(b_2 - 1) \,  \mathcal{F}^{(2)}_2 (b_1 + 1) -  (b_1 + \theta)  \, (b_2 + \phi - 1) \mathcal{F}^{(2)}_2 (b_2 - 1) = 0;\\
	& b_1 \,	(c_1 + \theta - 1) \,  \mathcal{F}^{(2)}_2 (b_1 + 1) -  (c_1 - 1)  \, (b_1 + \theta) \mathcal{F}^{(2)}_2 (c_1 - 1) = 0;\\
	& b_1 \,	(c_2 + \phi - 1) \,  \mathcal{F}^{(2)}_2 (b_1 + 1) -  (c_2 - 1)  \, (b_1 + \theta) \mathcal{F}^{(2)}_2 (c_2 - 1) = 0;\\
	& b_1 \,	c_1 \,  \mathcal{F}^{(2)}_2 (b_1 + 1) -  (c_1 + \theta)  \, (b_1 + \theta) \mathcal{F}^{(2)}_2 (c_1 + 1) = 0;\\
	& b_1 \,	c_2 \,  \mathcal{F}^{(2)}_2 (b_1 + 1) -  (c_2 + \phi)  \, (b_1 + \theta) \mathcal{F}^{(2)}_2 (c_2 + 1) = 0;\\
	& b_2 \,	(b_1 - 1) \,  \mathcal{F}^{(2)}_2 (b_2 + 1) -  (b_2 + \phi)  \, (b_1 + \theta - 1) \mathcal{F}^{(2)}_2 (b_1 - 1) = 0;\\
	& b_2 \,	(b_2 - 1) \,  \mathcal{F}^{(2)}_2 (b_2 + 1) -  (b_2 + \phi)  \, (b_2 + \phi - 1) \mathcal{F}^{(2)}_2 (b_2 - 1) = 0;\\
	& b_2 \,	(c_1 + \theta - 1) \,  \mathcal{F}^{(2)}_2 (b_2 + 1) -  (c_1 - 1)  \, (b_2 + \phi) \mathcal{F}^{(2)}_2 (c_1 - 1) = 0;\\
	& b_2 \,	(c_2 + \phi - 1) \,  \mathcal{F}^{(2)}_2 (b_2 + 1) -  (c_2 - 1)  \, (b_2 + \phi) \mathcal{F}^{(2)}_2 (c_2 - 1) = 0;\\
	& b_2 \,	c_1 \,  \mathcal{F}^{(2)}_2 (b_2 + 1) -  (c_1 + \theta)  \, (b_2 + \phi) \mathcal{F}^{(2)}_2 (c_1 + 1) = 0;\\
	& b_2 \,	c_2 \,  \mathcal{F}^{(2)}_2 (b_2 + 1) -  (c_2 + \phi)  \, (b_2 + \phi) \mathcal{F}^{(2)}_2 (c_2 + 1) = 0;\\
	& (b_2 - 1) \,	(b_1 + \theta - 1) \,  \mathcal{F}^{(2)}_2 (b_1 - 1) -  (b_1 - 1)  \, (b_2 + \phi - 1) \mathcal{F}^{(2)}_2 (b_2 - 1) = 0;\\
	& (b_1 + \theta - 1) \,	(c_1 + \theta - 1) \,  \mathcal{F}^{(2)}_2 (b_1 - 1) -  (c_1 - 1)  \, (b_1 - 1) \mathcal{F}^{(2)}_2 (c_1 - 1) = 0;\\
	& (b_1 + \theta - 1) \,	(c_2 + \phi - 1) \,  \mathcal{F}^{(2)}_2 (b_1 - 1) -  (c_2 - 1)  \, (b_1 - 1) \mathcal{F}^{(2)}_2 (c_2 - 1) = 0;\\
	& 	c_1 \, (b_1 + \theta - 1) \, \mathcal{F}^{(2)}_2 (b_1 - 1) - (b_1 - 1)  (c_1 + \theta)  \,  \mathcal{F}^{(2)}_2 (c_1 + 1) = 0;\\
	& 	c_2 \, (b_1 + \theta - 1) \, \mathcal{F}^{(2)}_2 (b_1 - 1) - (b_1 - 1)  (c_2 + \phi)  \,  \mathcal{F}^{(2)}_2 (c_2 + 1) = 0;\\
	& (b_2 + \phi - 1) \,	(c_1 + \theta - 1) \,  \mathcal{F}^{(2)}_2 (b_2 - 1) -  (c_1 - 1)  \, (b_2 - 1) \mathcal{F}^{(2)}_2 (c_1 - 1) = 0;\\
	& (b_2 + \phi - 1) \,	(c_2 + \phi - 1) \,  \mathcal{F}^{(2)}_2 (b_2 - 1) -  (c_2 - 1)  \, (b_2 - 1) \mathcal{F}^{(2)}_2 (c_2 - 1) = 0;\\
	& 	c_1 \, (b_2 + \phi - 1) \, \mathcal{F}^{(2)}_2 (b_2 - 1) - (b_2 - 1)  (c_1 + \theta)  \,  \mathcal{F}^{(2)}_2 (c_1 + 1) = 0;\\
	& 	c_2 \, (b_2 + \phi - 1) \, \mathcal{F}^{(2)}_2 (b_2 - 1) - (b_2 - 1)  (c_2 + \phi)  \,  \mathcal{F}^{(2)}_2 (c_2 + 1) = 0;\\
	& 	c_1 \, (c_1 - 1) \, \mathcal{F}^{(2)}_2 (c_1 - 1) - (c_1 + \theta - 1)  (c_1 + \theta)  \,  \mathcal{F}^{(2)}_2 (c_1 + 1) = 0;\\
	& 	 (c_1 - 1) \, (c_2 + \phi - 1) \, \mathcal{F}^{(2)}_2 (c_1 - 1) - (c_2 - 1)  \, (c_1 + \theta - 1)    \mathcal{F}^{(2)}_2 (c_2 - 1) = 0;\\
	& 	c_2 \, (c_1 - 1) \, \mathcal{F}^{(2)}_2 (c_1 - 1) - (c_1 + \theta - 1)  (c_2 + \phi)  \,  \mathcal{F}^{(2)}_2 (c_2 + 1) = 0;\\
	& 	 (c_1 + \theta) \, (c_2 + \phi - 1) \, \mathcal{F}^{(2)}_2 (c_1 + 1) - c_1 \, (c_2 - 1)  \,  \mathcal{F}^{(2)}_2 (c_2 - 1) = 0;\\
	& 	c_2 \, (c_1 + \theta) \, \mathcal{F}^{(2)}_2 (c_1 + 1) - c_1 \, (c_2 + \phi)  \,  \mathcal{F}^{(2)}_2 (c_2 + 1) = 0;\\
	& 	c_2 \, (c_2 - 1) \, \mathcal{F}^{(2)}_2 (c_2 - 1) - (c_2 + \phi - 1)  (c_2 + \phi)  \,  \mathcal{F}^{(2)}_2 (c_2 + 1) = 0.
\end{align}
Similarly using  the difference-differential relations, 
\begin{align}
	&	a \, \mathcal{F}^{(2)}_2 (a + 1)  = \left(a + \frac{1}{k} \Theta_{t} \right) \, \mathcal{F}^{(2)}_2;\\
	& \left(a + \frac{1}{k} \Theta_{t} - 1\right) \, \mathcal{F}^{(2)}_2 (a - 1)  =	(a - 1) \, \mathcal{F}^{(2)}_2;\\
	&	b_1 \, \mathcal{F}^{(2)}_2 (b_1 + 1)  = \left(b_1 + \theta\right) \, \mathcal{F}^{(2)}_2;\\
	&\left(b_1 + \theta - 1\right) \, \mathcal{F}^{(2)}_2 (b_1 - 1)  =	(b_1 - 1) \, \mathcal{F}^{(2)}_2;\\
	&	b_2 \, \mathcal{F}^{(2)}_2 (b_2 + 1)  = \left(b_2 + \phi\right) \, \mathcal{F}^{(2)}_2;\\
	& \left(b_2 + \phi - 1\right) \, \mathcal{F}^{(2)}_2 (b_2 - 1)  =	(b_2 - 1) \, \mathcal{F}^{(2)}_2;\\
	&	(c_1 - 1) \, \mathcal{F}^{(2)}_2 (c_1 - 1)  = \left(c_1 + \theta - 1\right) \, \mathcal{F}^{(2)}_2;\\
	& \left(c_1 + \theta\right) \, \mathcal{F}^{(2)}_2 (c_1 + 1)  =	c_1 \, \mathcal{F}^{(2)}_2;\\
	&	(c_2 - 1) \, \mathcal{F}^{(2)}_2 (c_2 - 1)  = \left(c_2 + \phi - 1\right) \, \mathcal{F}^{(2)}_2;\\
	& \left(c_2 + \phi\right) \, \mathcal{F}^{(2)}_2 (c_2 + 1)  =	c_2 \, \mathcal{F}^{(2)}_2,
\end{align}
one will get the following difference-differential recursion relations:
\begin{align}
	& a \, (a - 1) \, \mathcal{F}^{(2)}_2 (a + 1) - \left(a + \frac{1}{k} \Theta_{t}\right) \, \left(a + \frac{1}{k} \Theta_{t} - 1\right) \mathcal{F}^{(2)}_2 (a - 1) = 0;\\
	& a \, (b_1 - 1) \, \mathcal{F}^{(2)}_2 (a + 1)  - \left(a + \frac{1}{k} \Theta_{t} \right) \, \left(b_1 + \theta  - 1\right) \mathcal{F}^{(2)}_2 (b_1 - 1) = 0;\\
	& a \, (b_2 - 1) \, \mathcal{F}^{(2)}_2 (a + 1) - \left(a + \frac{1}{k} \Theta_{t}\right) \, \left(b_2 + \phi - 1\right) \mathcal{F}^{(1)}_1 (b_2 - 1) = 0;\\
	& a \, c_1 \, \mathcal{F}^{(2)}_2 (a + 1)  - \left(a + \frac{1}{k} \Theta_{t}\right) \, \left(c_1 + \theta \right) \mathcal{F}^{(2)}_2 (c_1 + 1) = 0;\\
	& a \, c_2 \, \mathcal{F}^{(2)}_2 (a + 1)  - \left(a + \frac{1}{k} \Theta_{t}\right) \, \left(c_2 + \phi \right) \mathcal{F}^{(2)}_2 (c_2 + 1) = 0;\\
	& a \, \left(b_1 + \theta\right) \, \mathcal{F}^{(2)}_2 (a + 1) - b_1 \, \left(a + \frac{1}{k} \Theta_{t}\right) \, \mathcal{F}^{(2)}_2 (b_1 + 1) = 0;\\
	& a \, \left(b_2 + \phi\right) \, \mathcal{F}^{(2)}_2 (a + 1) - b_2 \, \left(a + \frac{1}{k} \Theta_{t}\right) \, \mathcal{F}^{(2)}_2 (b_2 + 1) = 0;\\
	& a \, \left(c_1 + \theta - 1\right) \, \mathcal{F}^{(2)}_2 (a + 1)  - (c_1 - 1) \, \left(a + \frac{1}{k} \Theta_{t}\right) \, \mathcal{F}^{(2)}_2 (c_1 - 1) = 0;\\
	& a \, \left(c_2 + \phi - 1\right) \, \mathcal{F}^{(2)}_2 (a + 1)  - (c_2 - 1) \, \left(a + \frac{1}{k} \Theta_{t}\right) \, \mathcal{F}^{(2)}_2 (c_2 - 1) = 0;\\
	&\left(a + \frac{1}{k} \Theta_{t} - 1\right) \, \left(b_1 + \theta \right) \, \mathcal{F}^{(2)}_2 (a - 1) - b_1 \, (a - 1)  \, \mathcal{F}^{(2)}_2 (b_1 + 1) = 0;\\
	&	\left(a + \frac{1}{k} \Theta_{t} - 1\right) \, \left(b_2 + \phi\right) \, \mathcal{F}^{(2)}_2 (a - 1) - b_2 \, (a - 1)  \, \mathcal{F}^{(2)}_2 (b_2 + 1) = 0;\\
	&	\left(a + \frac{1}{k} \Theta_{t} - 1\right) \, \left(c_1 + \theta - 1\right) \, \mathcal{F}^{(2)}_2 (a - 1)  - (c_1 - 1) \, (a - 1)  \, \mathcal{F}^{(2)}_2 (c_1 - 1) = 0;\\
	&	\left(a + \frac{1}{k} \Theta_{t} - 1\right) \, \left(c_2 + \phi - 1\right) \, \mathcal{F}^{(2)}_2 (a - 1)  - (c_2 - 1) \, (a - 1)  \, \mathcal{F}^{(2)}_2 (c_2 - 1) = 0;\\
	& (b_1 - 1) \,	\left(a + \frac{1}{k} \Theta_{t} - 1\right) \,  \mathcal{F}^{(2)}_2 (a - 1)  -  (a - 1)  \, \left(b_1 + \theta  - 1\right) \mathcal{F}^{(2)}_2 (b_1 - 1) = 0;\\
	& (b_2 - 1) \,	\left(a + \frac{1}{k} \Theta_{t} - 1\right) \,  \mathcal{F}^{(2)}_2 (a - 1)  -  (a - 1)  \, \left(b_2 + \phi - 1\right) \mathcal{F}^{(2)}_2 (b_2 - 1) = 0;\\
	& c_1 \,	\left(a + \frac{1}{k} \Theta_{t} - 1\right) \,  \mathcal{F}^{(2)}_2 (a - 1) -  (a - 1)  \, \left(c_1 + \theta\right) \mathcal{F}^{(2)}_2 (c_1 + 1) = 0;\\
	& c_2 \,	\left(a + \frac{1}{k} \Theta_{t} - 1\right) \,  \mathcal{F}^{(2)}_2 (a - 1) -  (a - 1)  \, \left(c_2 + \phi\right) \mathcal{F}^{(2)}_2 (c_2 + 1) = 0.
\end{align}

\section{Conclusion} 
In the present paper, we continued the study of discrete analogues of Appell functions and worked out on the two distinct discrete analogues $\mathcal{F}^{(1)}_2$ and $\mathcal{F}^{(2)}_2$ of the Appell function $F_2$. The discrete forms of $F_2$ are not limited to these two functions, one can also get the other discrete forms of $F_2$. For example, if $k_1 = k = k_2$, then the discrete function $\mathcal{F}^{(1)}_2$ leads to the third discrete analogue of Appell function $F_2$ as
\begin{align}
	\mathcal{F}^{(3)}_2 & = \mathcal{F}^{(3)}_2(a, b_1, b_2; c_1, c_2; t_1, t_2, k, x, y)\nonumber\\ 
& = \sum_{m,n\geq0} \frac{(a)_{m+n} \, (b_1)_m \, (b_2)_n \, (-1)^{(m + n) k} \, (-t_1)_{mk} \,   (-t_2)_{nk}}{ (c_1)_{m} \, (c_2)_n\, m! \, n!} \ x^m \, y^n.
\end{align} 
We also introduced the discrete analogues of Humbert functions $\psi_1$ and $\psi_2$ as limiting cases of $\mathcal{F}^{(1)}_2$ and $\mathcal{F}^{(2)}_2$ but discrete Humbert functions can be studied independently. So there is an opportunity of investigating these discrete functions in comparative way and believe to yield new and interesting results.

The particular cases of the results obtained in this paper lead to the corresponding results for Appell function $F_2$ and  Kamp\'e de F\'eriet  hypergeometric function. Some results \emph{viz.} differential equations, differential formulas, integral representations, finite and infinite summation formulas along with others while the particular cases of first and second order difference-differential recurrence relations are believed to be new.

%\medskip
%{\bf Acknowledgments:} The authors thank the referee for valuable suggestions that led to a better presentation of the paper. The financial assistance provided to the first author in the form of a Junior Research Fellowship from Council of Scientific and Industrial Research, India is gratefully acknowledged.

\end{document}